\documentclass[11pt,a4paper,reqno]{amsart}
\usepackage[utf8]{inputenc}
\usepackage[T1]{fontenc}
\usepackage[UKenglish]{babel}
\usepackage[a4paper,margin=1in]{geometry}

\usepackage{amsmath,amsthm,amsfonts,amssymb}
\usepackage{mathrsfs,dsfont}
\usepackage{graphicx}
\usepackage{url}
\usepackage{verbatim}
\usepackage{xcolor}

\usepackage{hyperref}

\usepackage{marginnote}

\makeatletter
\def\namedlabel#1#2{\begingroup
    #2%
    \def\@currentlabel{#2}%
    \label{#1}\endgroup
}
\makeatother

\theoremstyle{plain}
\newtheorem{theorem}{Theorem}[section]
\newtheorem{corollary}[theorem]{Corollary}
\newtheorem{lemma}[theorem]{Lemma}
\newtheorem{proposition}[theorem]{Proposition}

\theoremstyle{definition}
\newtheorem{remark}[theorem]{Remark}
\newtheorem{example}[theorem]{Example}

\numberwithin{equation}{section}

\renewcommand\labelenumi{\textup{\alph{enumi})}}

\renewcommand\theenumi\labelenumi
\makeatletter\renewcommand{\p@enumii}{}\makeatother 

\renewcommand{\leq}{\leqslant}
\renewcommand{\le}{\leqslant}
\renewcommand{\geq}{\geqslant}
\renewcommand{\ge}{\geqslant}

\DeclareMathOperator{\spec}{spec}

\newcommand{\cD}{\mathcal{D}}

\newcommand{\R}{\mathds{R}}

\newcommand{\I}{\mathds{1}}

\newcommand{\N}{\mathds{N}}

\newcommand {\tki}[1]{{\widetilde{k}}^{(#1)}}

\newcommand {\kk}[1]{k^{(#1)}}

\def \D{\mathcal{D}}
\def \E{\mathcal{E}}

\def \Rdo {\mathbb{R}^d\setminus\{0\}}
\def \Rd {\R^d}

\def \NN {\mathbb{N}}
\def \tk{\widetilde{k}}
\def \tp{\widetilde{p}}
\def \tP{\widetilde{P}}
\def \eps{\varepsilon}
\def \nua{\nu^{(\alpha)}}
\def \pa{p^{(\alpha)}}
\def \Pa{P^{(\alpha)}}
\def \Ea{\E^{(\alpha)}}
\def \tpa{\widetilde{p}^{(\alpha)}}
\def \tPa{\widetilde{P}^{(\alpha)}}
\def \tE{\widetilde{\E}}
\def \tEa{\tE^{(\alpha)}}
\newcommand {\nur}{\nu}
\renewcommand{\d}{{\rm d}}
\newcommand{\spr}[1]{\left\langle #1 \right\rangle}
\newcommand{\norm}[2]{\big\| #1 \big\|_#2}

\begin{document}

\title[Fractional-type Schr\"odinger operators with singular potential]%
{Bound states and heat kernels for fractional-type Schr\"odinger operators with singular potentials}

\date{\today}

\author[T. Jakubowski]{Tomasz Jakubowski }
\address{
	Wroc{\l}aw University of Science and Technology
	Faculty of Pure and Applied Mathematics\\
	Wyb. Wyspia\'{n}skiego 27\\
	50-370 Wroc{\l}aw\\
	Poland
}
\email{tomasz.jakubowski@pwr.edu.pl}

\author[K. Kaleta]{Kamil Kaleta }
\address{
	Wroc{\l}aw University of Science and Technology
	Faculty of Pure and Applied Mathematics\\
	Wyb. Wyspia\'{n}skiego 27\\
	50-370 Wroc{\l}aw\\
	Poland
}
\email{kamil.kaleta@pwr.edu.pl}

\author[K. Szczypkowski]{Karol Szczypkowski }
\address{
	Wroc{\l}aw University of Science and Technology
	Faculty of Pure and Applied Mathematics\\
	Wyb. Wyspia\'{n}skiego 27\\
	50-370 Wroc{\l}aw\\
	Poland
}
\email{karol.szczypkowski@pwr.edu.pl}

\thanks{The first and the second author were partially supported by the grant 2015/18/E/ST1/00239 of National Science Centre, Poland.
The third author was partially supported through the DFG-NCN Beethoven Classic~3 programme,
contract no.\ 2018/31/G/ST1/02252 (National Science Center, Poland) and SCHI-419/111 (DFG, Germany).}

\begin{abstract}
We consider non-local Schr\"odinger operators $H=-L-V$ in $L^2(\R^d)$, $d \geq 1$, where the kinetic terms $L$ are pseudo-differential operators which are perturbations of the fractional Laplacian by bounded non-local operators and $V$ is the fractional Hardy potential. 
We prove pointwise estimates of eigenfunctions corresponding to negative eigenvalues and upper finite-time horizon estimates for heat kernels. We also analyze the relation between the matching lower estimates of the heat kernel and the ground state near the origin. Our results cover the relativistic Schr\"odinger operator with Coulomb potential.
\end{abstract}



\maketitle

\section{Introduction}

Let $d\in \N:=\{1,2,\ldots\}$ and $\alpha\in (0,2 \wedge d)$. Recall that the fractional Laplacian $L^{(\alpha)}:= -(-\Delta)^{\alpha/2}$ is a pseudo-differential operator which is defined by 
\[
\widehat{L^{(\alpha)} f}(\xi) = - |\xi|^\alpha \widehat{f}(\xi), \quad f \in \cD(L^{(\alpha)}) := \left\{f \in L^2(\R^d): |\xi|^\alpha \widehat f(\xi)\in L^2(\R^d) \right\},
\]
see e.g.\ Kwaśnicki \cite{MK-2017}. It is known that $|\xi|^\alpha=\int_{\Rd \setminus \left\{0\right\}} (1-\cos(\xi\cdot y))\nua(y)\d y$, $\xi\in \Rd$,
where
\[
\nua(y)
=c_{d,\alpha}|y|^{-d-\alpha},\quad y \in \Rd \setminus \left\{0\right\}, \quad \text{and} \quad
c_{d,\alpha}:=\frac{ \alpha 2^{\alpha-1}\Gamma\big((d+\alpha)/2\big)}{\pi^{d/2}\Gamma(1-\alpha/2)}.
\]
In this paper we consider the class of L\'evy operators $L$  such that 
\[
\widehat{L f}(\xi) = - \psi(\xi) \widehat{f}(\xi), \quad f \in \cD(L) := \left\{f \in L^2(\R^d): \psi \widehat f \in L^2(\R^d) \right\}
\]
(see Jacob \cite{Jacob} and B\"ottcher, Schilling and Wang \cite{Bottcher_Schilling_Wang}), where
$$
\psi(\xi) = \int_{\R^d \setminus\{0\}} (1-\cos(\xi \cdot y) \nu(y) \d y, \quad \xi \in \R^d,
$$
and the L\'evy density $\nu$ is symmetric (i.e.\ $\nu(-x) = \nu(x)$) and satisfies the following assumption: 

\medskip

	\begin{itemize}
\item[\bfseries(\namedlabel{A1}{A1})] 
    The density $\nu$ is such that
		\begin{align*}
	\sigma: = \nua - \nur \geq 0
	\end{align*}
	and $\sigma(dx) = \sigma(x) dx$ is a finite measure.
\end{itemize}

\medskip

The main goal of this paper is to give pointwise estimates of $L^2$-eigenfunctions corresponding to negative eigenvalues  (below we call them \emph{bound state} eigenfunctions or just \emph{bound states}) and finite-time horizon heat kernel estimates for non-local Schr\"odinger operators 
\begin{align} \label{eq:def_H}
H=-L-V, \quad \text{where} \quad V(x) = \frac{\kappa}{|x|^{\alpha}}, \quad 0 < \kappa < \kappa^*= \frac{2^\alpha \Gamma\left(\frac{d+\alpha}{4}\right)^2}{\Gamma\left(\frac{d-\alpha}{4}\right)^2}.
\end{align}
The operator $H$ is defined in form-sense as a bounded below self-adjoint operator, see Section \ref{sec:SchOp_def}.
An important example is the \emph{relativistic Schr\"odinger operator with Coulomb potential}
\[
H = \sqrt{-\Delta+m^2} -m - \frac{\kappa}{|x|}.
\]
Here we set $d=3$, $\alpha = 1$ and, consequently, $\kappa^*= 2/\pi$. The Hamiltonian $H+m$ is known to provide one of possible descriptions (neglecting spin effects) of the energy of a relativistic particle with mass $m$ in the Coulomb field. It has been widely studied as an alternative to the Klein--Gordon and the Dirac theories, see Herbst \cite{IH-1977}, Weder \cite{Weder-1975}, Daubechies and Lieb \cite{ID-EL-1983}, and Daubechies \cite{ID-1984}. The other examples of operators $L$ (L\'evy densities $\nu$) satisfying our assumption \eqref{A1} are given in Example \ref{ex:example_nu}. 

The potential $V$ in \eqref{eq:def_H} is the Hardy potential corresponding to fractional Laplacian $L^{(\alpha)}$ and $\kappa^*$ is known to be the critical constant in the fractional Hardy inequality
\begin{align} \label{eq:frac_Hardy}
\Ea[f] \geq \kappa^* \int_{\R^d} \frac{f(x)^2}{|x|^{\alpha}} \d x, \quad f \in L^2(\R^d);
\end{align}
$\Ea$ denotes the quadratic form of the operator $L^{(\alpha)}$, see Herbst \cite{IH-1977}, Frank and Seiringer \cite{RF-RS-2008}, Bogdan, Dyda and Kim \cite{KB-BD-PK-2016}. As shown in \cite[Section 4]{KB-BD-PK-2016} and \cite[Section 2.2]{KB-TG-TJ-DP-2019}, for any $0< \kappa \leq \kappa^*$ there exists a unique number $\delta$ such that
\begin{align} \label{eq:delta_def}
0 < \delta \leq \frac{d-\alpha}{2} \quad \text{and} \quad \kappa = \frac{2^\alpha \Gamma\left(\frac{\alpha+\delta}{2}\right) \Gamma\left(\frac{d-\delta}{2}\right)}{\Gamma\left(\frac{\delta}{2}\right)\Gamma\left(\frac{d-\alpha-\delta}{2}\right)}.
\end{align}

Our first main result gives pointwise estimates of eigenfunctions corresponding to negative eigenvalues for Schr\"odinger operators $H$ given by \eqref{eq:def_H}. It summarizes Proposition \ref{lem:est_at_zero} and Theorem \ref{th:EFest} which are proven below. Recall that the heat kernel $p(t,x,y):= p_t(y-x)$ and resolvent kernel $g_{\lambda}(x,y):= g_{\lambda}(y-x)$ of the operator $L$ are given by 
\[
p_t(x):=(2\pi)^{-d}\int_{\R^d} e^{-t\psi(\xi)}e^{-ix\cdot\xi}\d \xi, \quad t>0,\ x \in \R^d,
\] 
and
\[
g_{\lambda}(x) = \int_0^{\infty} e^{-\lambda t} p_t(x)\d t, \quad \lambda>0,\ x \in \R^d,
\]
see Section \ref{sec:free_operators} for more details. 

\begin{theorem}\label{th:bound_states_main}
Let \eqref{A1} hold and let $\kappa< \kappa^*$. Let $\varphi \in L^2(\R^d)$ be such that 
\begin{align} \label{eq:eigenequation}
H\varphi = E \varphi, \quad \text{for a number} \ E<0.
\end{align}
Then $\varphi$ has a version which is continuous on $\R^d \setminus \left\{0\right\}$ and satisfies the following estimates.
\begin{itemize}
\item[(a)] For every $R>0$ there exists $c>0$ such that
\[
|\varphi(x)| \leq c e^{|\sigma|+E} |x|^{-\delta}, \quad |x| \leq R, \ x \neq 0,
\]
where $\delta$ is determined by \eqref{eq:delta_def}; the constant $c$ depends neither on $\varphi$, $E$ nor $\sigma$. 

\item[(b)] For every $\eps \in (0,|E| \land 1)$ there is $R=R(\eps) \geq 1$ and $c=c(\eps)$ such that
\begin{align}\label{eqth:EFest_upper}
	|\varphi(x)| \le c \sup_{|y| \le R} g_{|E|-\eps}(x-y), \qquad x \in \Rd \setminus \left\{0\right\}. 
\end{align}
Furthermore, if $\varphi$ is a ground state (i.e.\ \eqref{eq:eigenequation} holds with $E = \inf \spec(H) <0$), then there is $\widetilde c>0$ such that
\begin{align}\label{eqth:EFest_lower}
	\varphi(x) \ge \widetilde c \inf_{|y| \le 1} g_{|E|}(x-y), \qquad x \in \Rd \setminus \left\{0\right\}. 
\end{align}
\end{itemize}
\end{theorem} 

Under the assumption \eqref{A1}, the operators $L$ and $L^{(\alpha)}$ are close to each other in the sense that their Fourier multipliers are asymptotically equivalent at infinity. Indeed, we have 
\[
\lim_{|\xi| \to \infty} \frac{\psi(\xi)}{|\xi|^{\alpha}} = 1,
\]
which is a direct consequence of
\[
 \psi(\xi) \leq |\xi|^{\alpha} = \psi(\xi) + \int_{\R^d \setminus\{0\}} (1-\cos(\xi \cdot y) \sigma(y) \d y \leq \psi(\xi) + |\sigma|, \quad \xi \in \R^d.
\]
However, the tail of $\nu(x)$ can be essentially lighter at infinity than that of $\nua(x)$, and therefore some analytic and spectral properties of $H=-L-V$ and $-L^{(\alpha)}-V$ may differ as well. The crucial example is the structure of the spectrum $\spec(H)$ of $H$. First recall that the discrete spectrum $\spec_{d}(H)$ consists of all isolated eigenvalues of $H$ of finite multiplicity with no accumulation points, and the essential spectrum $\spec_{e}(H)$ can be defined as $\spec_{e}(H):= \spec(H) \setminus \spec_{d}(H)$. It follows from \cite[Theorems 3.6, 3.7]{Weder-1975} (see also the more explicit statement in \cite[Proposition 4.1]{KK-JL-2016} which is dedicated to self-adjoint L\'evy operators with real-valued potentials) that $\spec_{e}(H) = \spec_{e}(-L) = [0,\infty)$. This implies that $\spec_{d}(H) \subset (-\infty, 0)$. Together with \eqref{eq:frac_Hardy}, it gives that $\spec_{d}(-L^{(\alpha)}-V) = \emptyset$. On the other hand, the discrete spectra $\spec_{d}(H)$ can be non-empty for Schr\"odinger operators $H$ as in \eqref{eq:def_H} where the fractional Laplacian $L^{(\alpha)}$ is replaced by the proper perturbation $L$. This is exactly the situation in which our Theorem \ref{th:bound_states_main} applies directly. 

Here, the key example is the relativistic Coulomb model which was mentioned above. It is known to produce infinitely many negative eigenvalues; the lowest eigenvalue is
nondegenerate, and the associated ground state is strictly positive, see \cite{IH-1977, Weder-1975, ID-EL-1983}. The $L^2$-upper estimates of eigenfunctions for this model have been obtained by Nardini \cite{FN-1986} (see also \cite{FN-1988}). Our present Theorem \ref{th:bound_states_main} gives pointwise estimates at infinity and at zero, and together with Theorem \ref{th:heat_kernel_and_ground_state} which is stated below, it also gives two-sided bounds for the ground state. To the best of our knowledge, such a result has not been known before. Precise statements and the detailed discussion of this example are postponed to Section \ref{sec:relativistic} below. We want to mention here that the assumption \eqref{A1} is motivated by the paper of Ryznar \cite{MR-2002}, who proposed such an approach in the study of the potential theory of relativistic stable processes. At the level of operators this idea was one of basic tools in the study of the stability of relativistic matter -- in that theory, for some estimates, the relativistic operator $\sqrt{-\Delta+m^2} -m$ can be replaced by its bounded perturbation $\sqrt{-\Delta}$, see Fefferman and de la Llave \cite{CF-RL-1986}, Lieb and Yau \cite{EL-YH-1988}, Frank, Lieb and Seiringer \cite{RF-EL-RS-2007}, and the monograph by Lieb and Seiringer \cite{EL-RS-2010}. We also refer to the recent paper by Ascione and L\H orinczi \cite{GA-JL-2022} for an application to the zero energy inverse problems for relativistic Schr\"odinger operators.

The decay rates of bound states at infinity for Schr\"odinger operators involving the Euclidean Laplacian are now a classical topic, see Agmon \cite{SA-1982,SA-1985}, Reed and Simon \cite{MR-BS-1978}, and Simon \cite{BS-1982}. In the non-local setting this problem has been studied for Schr\"odinger operators with less singular decaying potentials, see \cite{RC-WM-BS-1990,KK-JL-2016} (see also \cite{KK-JL-2020} for estimates in the zero energy case). Pointwise estimates as in \eqref{eqth:EFest_upper} and \eqref{eqth:EFest_lower} have been first established by Carmona, Masters and Simon in their seminal paper \cite{RC-WM-BS-1990}. The authors identified the decay rates at infinity of eigenfunctions corresponding to negative eigenvalues for a large class of L\'evy operators $L$ perturbed by Kato-decomposable potentials $V=V_{+}-V_{-}$ (also called Kato-Feller potentials, see \cite{Demuth_Casteren}) for which the negative part $V_{-}$ is in the Kato class associated with $L$, and the positive part $V_{+}$ is locally in that class. Recall that if $V(x)=\kappa |x|^{-\beta}$, for $\kappa, \beta >0$, then $V$ belongs to the Kato class of $L$ if and only if $\beta < \alpha$, see e.g.\ \cite[Section 3]{TG-KK-PS-2022}. Therefore, the singularity of the potential in \eqref{eq:def_H} is critical for these operators. Part (b) of Theorem \ref{th:bound_states_main} states that the estimates of Carmona, Masters and Simon extend to such a case, i.e.\ the criticality of the potential does not change the decay rates of bound states at infinity. The difference is that in our present setting the eigenfunctions can be singular at zero, while for the Kato class potentials they are typically bounded and continuous on $\R^d$. Part (a) of Theorem \ref{th:bound_states_main} gives the upper estimate of bound states around zero for potentials with critical singularity. The matching lower bound for the gound state is discussed in Theorem \ref{th:heat_kernel_and_ground_state} below. Such pointwise estimates for L\'evy operators with the fractional Hardy potential have not been known before. We remark that the estimates of the resolvent kernels for some special examples of L\'evy operators can be extracted e.g.\ from \cite{RC-WM-BS-1990,KK-JL-2016}. 

The methods of the paper \cite{RC-WM-BS-1990} are probabilistic -- they are based on the Feynman--Kac formula, martingales and the probabilistic potential theory. These tools are not available for more singular potentials as that in \eqref{eq:def_H}, see the monograph by Demuth and van Casteren \cite{Demuth_Casteren} for a systematic introduction to the stochastic spectral analysis of Feller operators. It makes the problem we treat in this paper much more difficult. We propose a new, fully analytic approach, which is completely different. The methods we use are based on the technique of perturbations of integral kernels developed by Bogdan, Hansen and Jakubowski \cite{KB-WH-TJ-2008}. We first construct the kernel $\tp(t,x,y)$ which is a perturbation of the heat kernel $p(t,x,y)$ by Hardy potential $V$, and analyze the smoothing properties of the corresponding semigroup of operators $\{\tP_t\}_{t>0}$ (Section \ref{sec:pert}). In Section \ref{sec:q_forms} we study the properties of quadratic forms. We describe in detail the relation between forms of $L$ and $L^{(\alpha)}$, and show that the form of the Schr\"odinger operator $H$ can be identified with quadratic form of the semigroup $\{\tP_t\}_{t>0}$. Consequently, by uniqueness, $e^{-tH} = \tP_t$, $t>0$ on $L^2(\R^d)$. In particular, $e^{-tH}$ are integral operators with kernels $\tp(t,x,y)$, which is crucial for our further investigations. Indeed, it provides direct access to properties of operators $e^{-tH}$. This is given in Theorem \ref{th:forms_equal}. Finally, all these partial results are applied in Section \ref{sec:ef_est} to establish the pointwise estimates for bound states. This part starts with the perturbation formula and combines some direct observations with the self-improving estimate which involves the convolutions of resolvent kernels. Our proofs do not require sharp estimates of the kernel $\tp(t,x,y)$. On the other hand, we use in an essential way the estimates of the heat kernel for the fractional Laplacian with Hardy potential which have been obtained just recently by Bogdan, Grzywny, Jakubowski and Pilarczyk \cite{KB-TG-TJ-DP-2019}. 

The second part of the paper is devoted to estimates of the kernel $\tp(t,x,y)$. 
Here, and in what follows, we write $f \asymp g$ on $A$, if $f,g \geq 0$ on $A$ and there is a (comparability) constant $c \geq 1$ such that
$c^{-1} g \leq f \leq c g$ holds on $A$ (``$\stackrel{c}{\asymp}$'' means the comparison with the constant $c$).
Also, we use ``$:=$'' to indicate the definition.
As usual, $a \land b := \min\{a,b\}$, $a \vee b := \max\{a,b\}$. 

In order to get sharp results, we need to impose the following additional assumption on the L\'evy density $\nu$.

\medskip

	\begin{itemize}
\item[\bfseries(\namedlabel{A2}{A2})] 
    There exists a non-decreasing profile function $f \colon (0,\infty) \to (0,\infty)$ such that
	\begin{align} \label{eq:profile}
	\nur(x) \asymp f(|x|), \qquad x \in \Rdo,
	\end{align}
	and the condition
	\begin{align} \label{eq:DJP}
	\sup_{x \in \R^d} \int_{\R^d} \frac{f_1(|x-y|)f_1(|y|)}{f_1(|x|)} < \infty,
	\end{align}
	where $f_1:= f \wedge 1$, holds. 
\end{itemize}

\medskip
As proved by Kaleta and Sztonyk \cite{KK-PS-2017}, the condition \eqref{eq:DJP} provides a minimal regularity of the profile $f$ which is required for comparability $p_t(x) \asymp t \nu(x)$, for large $x \in\R^d$. Easy-to-check sufficient (and necessary) conditions for it can be found in \cite[Section 3.1]{KK-RS-2020} and  \cite[Proposition 2]{KK-PS-2017}. In particular, all the examples of L\'evy densities $\nu$ presented in Example \ref{ex:example_nu} satisfy \eqref{A2}. 

Our second main result is the following theorem which gives the finite-time horizon upper estimate of $\tp(t,x,y)$. The construction of this kernel in Section  \ref{sec:pert} was performed for the full range of $\kappa \leq \kappa^*$. Our next theorem also covers the critical case $\kappa = \kappa^*$.

\begin{theorem} \label{th:heat_kernel}
Let \eqref{A1} and \eqref{A2} hold and let $\kappa \leq \kappa^*$. Then for every $T>0$ there exists a constant $c=c(T)$ such that
\[
\tp(t,x,y) \leq c \left(1+\frac{t^{\delta/\alpha}}{|x|^{\delta}} \right)\left(1+\frac{t^{\delta/\alpha}}{|y|^{\delta}} \right) p(t,x,y), \quad x,y \in \R^d \setminus \left\{0\right\}, \ t \in (0,T].
\]
\end{theorem}
It is clear that $\tp(t,x,y) \geq p(t,x,y)$, $x, y \in \R^d$, $t >0$. This means that the estimate in the above theorem is sharp at least for the case when both $|x|$ and $|y|$ are large. One can conjecture that the two-sided estimate as in Theorem \ref{th:heat_kernel} holds true for the full range of spatial variables. However, a general argument leading to the lower estimate in our present generality seems to be not available at the moment. On the other hand, such a conjecture is supported by our results in \cite{TJ-KK-KS-2022}, where we have proven the two-sided estimate of this form for the case when $L=-(-\Delta+m^{2/\alpha})^{\alpha/2}+m$, $\alpha \in (0,2)$, $m>0$. This is an example of the operator covered by our assumptions \eqref{A1}-\eqref{A2}. Such two-sided bounds, for $t>0$, have been first obtained for the fractional Laplacian $L^{(\alpha)}$ with Hardy potential \cite{KB-TG-TJ-DP-2019}. The proof of Theorem \ref{th:heat_kernel} uses in an essential way the results from \cite{KB-TG-TJ-DP-2019}, the properties of the density $\nu(x)$ and the sharp bound
\begin{align}\label{eq:pt_est_intro}
p(t,x,y) \asymp t^{-d/\alpha} \wedge t\nu(y-x), \quad t \in (0,T], \ x,y \in \R^d,
\end{align}
which was established in \cite{KK-PS-2017} (see Lemma \ref{lem:properties_pt_nu} below for details). Observe that \eqref{eq:pt_est_intro} makes the estimate in Theorem \ref{th:heat_kernel} explicit. 

Our last result relates the lower estimate of the kernel $\tp(t,x,y)$ with the lower bound of the ground state at zero for $\kappa <\kappa^*$.

\begin{theorem} \label{th:heat_kernel_and_ground_state}
Let \eqref{A1} and \eqref{A2} hold and let $\kappa <\kappa^*$. Let $\delta$ be the number determined by \eqref{eq:delta_def}. Assume that there exists $\varphi \in L^2(\R^d)$ such that $H \varphi = E \varphi$, where $E = \inf \sigma(H) <0$, i.e.\ the ground state of $H$ exists. Then the following statements are equivalent.
\begin{itemize}
\item[(a)] For every $T>0$ there exists a constant $c=c(T)$ such that
\[
\tp(t,x,y) \geq c \left(1+\frac{t^{\delta/\alpha}}{|x|^{\delta}} \right)\left(1+\frac{t^{\delta/\alpha}}{|y|^{\delta}} \right) p(t,x,y), \quad |x| \wedge |y| \leq t^{1/\alpha}, \ x,y \neq 0, \ t \in (0,T].
\]

\item[(b)] For every $R>0$ there exists a constant $c>0$ such that
\[
\varphi(x) \geq c|x|^{-\delta}, \quad |x| \leq R, \ x \neq 0.
\]
\end{itemize}
\end{theorem}

As mentioned above, for $|x|,|y| \geq t^{1/\alpha}$ we always have
\[
\tp(t,x,y) \geq p(t,x,y) \geq \frac{1}{4} \left(1+\frac{t^{\delta/\alpha}}{|x|^{\delta}} \right)\left(1+\frac{t^{\delta/\alpha}}{|y|^{\delta}} \right) p(t,x,y),
\]
which means that in fact the statement (b) implies the lower bound for the kernel $\tp(t,x,y)$ for the full range of time-space variables.
In combination with our results in \cite{TJ-KK-KS-2022}, Theorem \ref{th:heat_kernel_and_ground_state} gives the matching lower bound for the ground state of the relativistic Coulomb model. Details are given in Section \ref{sec:relativistic} below. Theorem \ref{th:heat_kernel_and_ground_state} is proved is Section \ref{sec:last_result}. We remark that the sharp two-sided finite-time horizon estimates for heat kernels of fractional Laplacian and more general L\'evy operators (covering our assumptions \eqref{A1}--\eqref{A2}) with Kato class potentials can be found in Bogdan, Hansen and Jakubowski \cite{KB-WH-TJ-2008} and Grzywny, Kaleta and Sztonyk \cite{TG-KK-PS-2022}.

\section{Schr\"odinger perturbations of heat kernels} \label{sec:pert}

\subsection{General framework} \label{sec:kernel_pert}
Our approach in this paper uses the technique of perturbation of kernels which was developed in \cite{KB-WH-TJ-2008}. We start with abstract settings.

Let $(E,\mathcal{F}, \mu)$ be a $\sigma$-finite measure space. Let $k : (0,\infty) \times E\times E \to [0,\infty)$ be a $\mathcal{B}((0,\infty)) \times \mathcal{F} \times \mathcal{F}$-measurable kernel such that
\begin{align} 
&\int_E k(s,x,z)k(t,z,y) \mu(\d z) = k(t+s,x,y), &&  x,y \in E, \label{eq:Ch-K}\\
&k(t,x,y) = k(t,y,x)\ge 0, &&  t>0,\; x,y \in E,\label{eq:symmetry}
\end{align}
and let $V \colon E \to [0,\infty)$ be an $\mathcal{F}$-measurable function. We define
\begin{align} \label{eq:pert_series}
\tk(t,x,y) = \sum_{n=0}^\infty k_{n}(t,x,y),
\end{align}
where
\begin{align*}
k_{0}(t,x,y) &= k(t,x,y) \\
k_{n}(t,x,y) &= \int_0^t \int_E k(s,x,z) V(z) k_{n-1}(t-s,z,y)\mu(\d z) \d s, \qquad n\ge1.
\end{align*}
It is known that $\tk(t,x,y)$ is a symmetric transition density, i.e.\ it satisfies the equalities \eqref{eq:Ch-K}, \eqref{eq:symmetry}, see \cite{KB-WH-TJ-2008}. Moreover, the following  perturbation formula 
\begin{align*}
\tk(t,x,y) & = k(t,x,y) + \int_0^t \int_E k(s,x,z) V(z) \tk(t-s,z,y)\mu(\d z) \d s \\
           & = k(t,x,y) + \int_0^t \int_E \tk(s,x,z) V(z) k(t-s,z,y)\mu(\d z) \d s, \quad t>0, \ x,y \in E,
\end{align*}
holds. The starting point of our investigations is the following direct observation. 

\begin{lemma}\label{lem:simpleest} 
Suppose we are given two kernels $\kk{1}(t,x,y)$ and $\kk{2}(t,x,y)$ as above and denote by $\tki1(t,x,y)$ and $\tki2(t,x,y)$ the corresponding perturbed kernels as in \eqref{eq:pert_series}. If
\begin{align} \label{eq:le}
\kk{1}(t,x,y) \le \kk{2}(t,x,y), \quad t>0, \ x,y \in E.
\end{align}
then
\begin{align} \label{eq:le_pert}
\tki1(t,x,y) \le \tki2(t,x,y), \quad t>0, \ x,y \in E.
\end{align}
\end{lemma}

\begin{proof}
It follows from \eqref{eq:pert_series}, \eqref{eq:le} and the assumption $V \geq 0$. 
\end{proof}

This fact will be used below in the following form.

\begin{corollary}\label{cor:tildepest}
Assume that there is $\lambda>0$ such that
\begin{align}
	&\kk1(t,x,y) \le e^{\lambda t}\kk2(t,x,y) && t>0,\; x,y \in E. \label{eq:le1}
\end{align}
Then
\begin{align*}
	&\tki1(t,x,y) \le e^{\lambda t}\tki2(t,x,y) && t>0,\; x,y \in E.
\end{align*}
\end{corollary}
\begin{proof}
We apply Lemma \ref{lem:simpleest} to the kernel $k(t,x,y):= e^{\lambda t} \kk2(t,x,y)$ instead of $\kk2(t,x,y)$. We only have to make sure that $\tk(t,x,y) = e^{\lambda t} \tki2(t,x,y)$. To this end, we define $k_{0}(t,x,y) = k(t,x,y)$,
\begin{align*}
k_{n}(t,x,y) &= \int_0^t \int_E k(s,x,z) q(z) k_{n-1}(t-s,x,y)\mu(\d z) \d s, \qquad n\ge1,
\end{align*}
and observe that $k_{n}(t,x,y) = e^{\lambda t}\kk2_{n}(t,x,y)$, which yields the desired assertion.
\end{proof}

\subsection{Heat kernels of free operators} \label{sec:free_operators}

Let $d\in \N:=\{1,2,\ldots\}$, $\alpha\in (0,2)$ and $\alpha<d$. Let 
$$
\nua(y)
=c_{d,\alpha}|y|^{-d-\alpha},\quad y \in \Rd \setminus \left\{0\right\}, 
$$
where
$$
c_{d,\alpha}:=\frac{ \alpha 2^{\alpha-1}\Gamma\big((d+\alpha)/2\big)}{\pi^{d/2}\Gamma(1-\alpha/2)}.
$$
This coefficient is so chosen that 
\begin{equation}
  \label{eq:trf}
  \int_{\Rd} (1-\cos(\xi\cdot y))\nua(y)\d y=|\xi|^\alpha, \quad \xi\in \Rd.
\end{equation}
Let $\nu(dx) = \nu(x)dx$ be a symmetric L\'evy measure such that \eqref{A1} holds. Consider
$$
\psi(\xi) = \int_{\R^d \setminus\{0\}} (1-\cos(\xi \cdot y) \nu(dy), \quad \xi \in \R^d.
$$
Observe that we have
$$
|\xi|^{\alpha} - |\sigma| \leq \psi(\xi) \leq |\xi|^{\alpha}, \quad \xi \in \R^d,
$$
which gives 
\begin{align} \label{eq:LK_comp}
\psi(\xi) \asymp |\xi|^{\alpha}, \quad |\xi| \geq (2|\sigma|)^{1/\alpha}. 
\end{align}
In particular, $e^{-t \psi(\cdot)} \in L^1(\R^d)$, for every $t>0$.
We define
\begin{equation}
  \label{eq:free_density_stable}
  \pa_t(x):=(2\pi)^{-d}\int_{\R^d} e^{-t|\xi|^\alpha}e^{-ix\cdot\xi}\d \xi, \quad t>0,\ x \in \R^d,
\end{equation}
and
\begin{equation}
  \label{eq:free_density}
  p_t(x):=(2\pi)^{-d}\int_{\R^d} e^{-t\psi(\xi)}e^{-ix\cdot\xi}\d \xi, \quad t>0,\ x \in \R^d.
\end{equation}
For every $t>0$, $\pa_t$ and $p_t$ are continuous and bounded probability density functions, see e.g. \cite{VK-RS-2013}. Moreover, $\pa_t$ is radial and $p_t$ is symmetric, i.e.\ $p_t(-y)=p_t(y)$, $y \in \R^d$. 
From \eqref{eq:free_density_stable} we have
\begin{equation}
  \label{eq:sca}
  \pa_t(x)=t^{-d/\alpha}\pa_1(t^{-{1/\alpha}}x)\,.
\end{equation}
It is also well-known that 
\begin{equation}\label{eq:oppt}
\pa_t(x) \asymp t^{-d/\alpha}\land \frac{t}{|x|^{d+\alpha}}
\,,\quad t>0,\,x\in \Rd\,.
\end{equation}
We denote 
$$
\pa(t,x,y)=\pa_t(y-x), \quad p(t,x,y)=p_t(y-x), \quad t>0,\ x,y\in \Rd.
$$
Clearly, $\pa_t$ and $p_t$ are symmetric kernels satisfying the Chapman-Kolmogorov equations:
\begin{equation}\label{eq:ChKforpa}
     \int_{\Rd} \pa(s,x, y)\pa(t,y, z)\d y = \pa(t+s,x, z), \quad x, z \in  \Rd,\, s, t > 0,
\end{equation}    
\begin{equation}\label{eq:ChKforp}
     \int_{\Rd} p(s,x, y)p(t,y, z)\d y = p(t+s,x, z), \quad x, z \in  \Rd,\, s, t > 0.
\end{equation}  
We let
$$
P^{(\alpha)}_t f(x) = \int_{\R^d} f(y) \pa(t,x,y) dy, \quad f \in L^2(\R^d,dx), \ t>0,
$$ 
$$
P_t f(x) = \int_{\R^d} f(y) p(t,x,y) dy, \quad f \in L^2(\R^d,dx), \ t>0;
$$
$(P_t^{(\alpha)})_{t > 0}$, $(P_t)_{t > 0}$ are strongly continuous semigroups of self-adjoint contractions on $L^2(\R^d,dx)$, and 
the operators $L^{(\alpha)}$ and $L$ are $L^2$-generators of these semigroups, see e.g.\ \cite[vol.\ 1, Example 4.7.28, pp.\ 407--409]{Jacob}. The kernels $\pa(t,x,y)$, $p(t,x,y)$ are heat kernels of the operators $L^{(\alpha)}$, $L$, respectively. 

Since $\nu^{(\alpha)} = \nu + \sigma$, we have
\begin{align*}
p^{(\alpha)}_t(x) = p_t * \mu^\sigma_t(x) = \int_{\R^d} p_t(x-y) \mu^\sigma_t(\d{y}), \quad x \in \R^d, \ t>0,
\end{align*}
where
\begin{align*}
\mu^\sigma_t(\d{x}) = e^{-|\sigma|t} \sum_{k=0}^\infty \frac{t^k \sigma^{k*}(\d{x})}{k!}
 = e^{-|\sigma|t} \delta_0(\d{x}) + e^{-|\sigma|t} \sum_{k=1}^\infty \frac{t^k \sigma^{k*}(x)\d{x}}{k!}.
\end{align*}
Hence,
\begin{align} \label{eq:relation_kernels}
p^{(\alpha)}_t(x) = e^{-|\sigma|t} p_t(x) + e^{-|\sigma|t} \sum_{k=1}^\infty \frac{t^k (p_t*\sigma^{k*})(x)}{k!}
\end{align}
and, consequently, 
\begin{align}\label{eq:domination_free_kernels}
p(t,x,y) \le e^{|\sigma|t} p^{(\alpha)}(t,x,y), \quad x, y \in \R^d, \ t>0.
\end{align}

We close this section by collecting some examples of L\'evy measures $\nu$ satisfying \eqref{A1} and \eqref{A2}.

\begin{example} \label{ex:example_nu}
\begin{enumerate}
\item[(a)] \textit{Relativistic stable L\'evy density}: let $\alpha \in (0,2)$ and $m>0$ and let 
\begin{align*}
\nu(x) & = \frac{ \alpha(4\pi)^{d/2}}{2\Gamma(1-\alpha/2)} \int_0^{\infty} \exp\left(-\frac{|x|^2}{4u} - m^{2/\alpha} u\right) u^{-1-\frac{d+\alpha}{2}} \d u \\
       & = \frac{\alpha 2^{\frac{\alpha-d}{2}} m^{\frac{d+\alpha}{2\alpha}}}{\pi^{\frac{d}{2}}\Gamma(1-\alpha/2)}
            \frac{K_{\frac{d+\alpha}{2}}\big(m^{\frac{1}{\alpha}}|x|\big)}{|x|^{\frac{d+\alpha}{2}}} , \quad x \in \R^d \setminus \left\{0\right\},
\end{align*} 
where 
\begin{align*}
K_{\mu}(r) = \frac{1}{2} \left(\frac{r}{2}\right)^{\mu} \int_0^{\infty} u^{-\mu-1} \exp\left(-u-\frac{r^2}{4u}\right) \d u, \quad \mu>0, \ r>0,
\end{align*} 
is the modified Bessel function of the second kind, see e.g.\ \cite[10.32.10]{NIST}.
As proved in \cite[Lemma 2]{MR-2002}, 
\[
\sigma(x) := \nu(x) - \nua(x)
\]
is a positive density of a finite measure such that
\[
|\sigma|=\int_{\R^d} \sigma(x) \d x = m \quad \text{and} \quad \sigma(x) \leq \frac{c}{|x|^{d+\alpha-2}}, \quad x \in \R^d \setminus \left\{0\right\},
\]
for some constant $c>0$. It is then clear that \eqref{A1} holds. In order to verify \eqref{A2}, we first observe that the density $\nu(x)$ is radial decreasing function. By using the well-known asymptotics,
\[
\lim_{r \to \infty} K_{\mu}(r) \sqrt{r} e^r  = \sqrt{\pi/2},
\]
we can also show that
\[
\nu(x) \asymp e^{-m^{1/\alpha}|x|} |x|^{-\frac{d+\alpha+1}{2}}, \quad |x| \geq 1.
\]
Consequently, \eqref{A2} holds by combination of \cite[Proposition 2]{KK-PS-2017} and \cite[Lemma 3.1]{KK-RS-2020}. 

The L\'evy operator associated to $\nu$ is the \emph{relativistic ($\alpha$-stable) operator}
\[
L=-(-\Delta+m^{2/\alpha})^{\alpha/2}+ m. 
\]

\smallskip

\item[(b)] \textit{Tempered stable L\'evy density:} let $\beta > \alpha$ and $\lambda>0$, and let
\begin{align*}
\nu(x) = e^{-\lambda |x|^{\beta}} \nua(x) , \quad x \in \R^d \setminus \left\{0\right\}.
\end{align*}
Then
\begin{align*}
\sigma(x) = \nua(x)- \nu(x) = c_{d,\alpha}|x|^{-d-\alpha} \left(1 - e^{-\lambda|x|^{\beta}}\right),
\end{align*} 
where
\[
c_{d,\alpha}=\frac{ \alpha 2^{\alpha-1}\Gamma\big((d+\alpha)/2\big)}{\pi^{d/2}\Gamma(1-\alpha/2)},
\]
is a positive density of a finite measure. By direct calculations,
\begin{align*}
|\sigma|= \int_{\R^d} \sigma(x) \d x & = c_{d,\alpha} \frac{ 2  \pi^{d/2}}{\Gamma\big(\frac{d}{2}\big)}\int_0^{\infty} (1-e^{-\lambda u^{\beta}}) u^{-1-\alpha} \d u \\
                         & = \frac{ \alpha 2^{\alpha}\Gamma\big(\frac{d+\alpha}{2}\big)}{\beta \Gamma\big(\frac{d}{2}\big) \Gamma\big(1-\frac{\alpha}{2}\big)}\int_0^{\infty} (1-e^{-\lambda u}) u^{-1-\frac{\alpha}{\beta}} \d u
												  =\frac{ 2^{\alpha}\Gamma\big(\frac{d+\alpha}{2}\big)\Gamma\big(1-\frac{\alpha}{\beta}\big)}{\Gamma\big(\frac{d}{2}\big) \Gamma\big(1-\frac{\alpha}{2}\big)} \lambda^{\frac{\alpha}{\beta}},
\end{align*}
where in the last line we have used the standard identity
\[
\frac{\gamma}{\Gamma(1-\gamma)} \int_0^{\infty} (1-e^{-\lambda u}) u^{-1-\gamma} \d u = \lambda^{\gamma}, \quad \lambda>0, \ \gamma \in (0,1),
\]
see e.g.\ \cite[(1)]{RS-RS-ZV-2012}. 
In particular, \eqref{A1} holds. Moreover, it follows from \cite[Proposition 2]{KK-PS-2017} that \eqref{A2} is satisfied if and only if $\alpha \in (0,1)$ and $\beta \in (\alpha,1]$. 

\smallskip

\item[(c)] \textit{Stable L\'evy density suppressed on a complement of a neighbourhood of the origin:} let $\eta:\R^d \to (0,1]$ be a function such that there exists a decreasing profile $g:(0,\infty) \to (0,\infty)$ such that $\eta(x) \asymp g(|x|)$, $x \in \R^d \setminus \left\{0\right\}$. Let $r>0$ and let 
\[
\nu(x) = \nua(x) \I_{\left\{|x| \leq r \right\}} + \eta(x) \nua(x) \I_{\left\{|x| > r \right\}}, \quad x \in \R^d \setminus \left\{0\right\}.
\]
Then
\[
\sigma(x) = \nua(x)- \nu(x) = \eta(x) \nua(x) \I_{\left\{|x| > r \right\}} \geq 0
\]
is a density of a finite measure. Indeed, 
\[
|\sigma|= \int_{\R^d} \sigma(x) \d x = \int_{|x| > r} (1-\eta(x)) \nua(x) \d x \leq \int_{|x| > r} \nua(x) \d x < \infty.
\]
Hence, \eqref{A1} holds. If we moreover assume that there exists a constant $c>0$ such that 
\[
g(s)g(t) \leq c g(s+t), \quad s, t \geq 1,
\]
then \eqref{A2} holds with the profile $f(r)= g(r)r^{-d-\alpha}$. 

An example of such $\nu$ is the \emph{layered stable L\'evy density} of the form 
\[ \nu(x) = c_{d,\alpha} |x|^{-d-\alpha}(1 \wedge |x|^{-\gamma}), \quad x \in \R^d \setminus \left\{0\right\}, \]
where $\gamma>0$.
\end{enumerate}

\end{example}

\subsection{Kernels perturbed by fractional Hardy potential}  \label{sec:pert_kernel}
Let $d \in \N$ and $0 < \alpha < 2 \wedge d$. By following \cite{KB-TG-TJ-DP-2019}, we consider the function $[0,(d-\alpha)/2] \ni \beta \longmapsto \kappa_{\beta}$ given by $\kappa_0=0$ and $g$ satisfies 
$$
\kappa_{\beta}
:= \frac{2^\alpha \Gamma\left(\frac{\alpha+\beta}{2}\right) \Gamma\left(\frac{d-\beta}{2}\right)}{\Gamma\left(\frac{\beta}{2}\right)\Gamma\left(\frac{d-\alpha-\beta}{2}\right)}, \quad 0< \beta \leq \frac{d-\alpha}{2}. 
$$
It is a strictly increasing and continuous function with maximal (critical) value
$$
\kappa^*=\kappa_{(d-\alpha)/2} = \frac{2^\alpha \Gamma\left(\frac{d+\alpha}{4}\right)^2}{\Gamma\left(\frac{d-\alpha}{4}\right)^2}. 
$$
It is then clear that it establishes a one-to-one correspondence between numbers in $[0,(d-\alpha)/2]$ and $[0,\kappa^*]$. 
\textbf{Throughout the paper we assume that}

\begin{align} \label{eq:kappa_and_delta}
\textbf{ $\kappa \in [0,\kappa^*]$ is fixed and $\delta \in [0,(d-\alpha)/2]$ is such that $\kappa_{\delta} = \kappa$}\,. 
\end{align}

\medskip

Let
$$
V(x) = V_{\kappa}(x) = \frac{\kappa}{|x|^{\alpha}}.
$$
When $\kappa < \kappa^*$, then $V_{\kappa}$ is said to be \emph{sub-critical}.

Starting with $k(t,x,y)=p(t,x,y)$ we can now construct the perturbed kernel $\tp(t,x,y)$ according to the procedure described in Section \ref{sec:kernel_pert}. Recall that $\tp(t,x,y)$ is a symmetric transition density such that 
\begin{align} \label{eq:perturbation}
\tp(t,x,y) & = p(t,x,y) + \int_0^t \int_{\R^d} p(s,x,z) V(z) \tp(t-s,z,y)\d z \d s \\
           & = p(t,x,y) + \int_0^t \int_{\R^d} \tp(s,x,z) V(z) p(t-s,z,y)\d z \d s, \quad t>0, \ x,y \in \R^d. \nonumber
\end{align}
Such a construction has been recently performed in \cite{KB-TG-TJ-DP-2019} for the case when $k(t,x,y)=\pa(t,x,y)$. It was proved in \cite[Theorems 1.1 and 3.1]{KB-TG-TJ-DP-2019} that
\begin{align}\label{eq:tilde_pa_est}
\tpa(t,x,y) \asymp \left(1+\frac{t^{\delta/\alpha}}{|x|^{\delta}}\right)\left(1+\frac{t^{\delta/\alpha}}{|y|^{\delta}}\right) \left(t^{-d/\alpha}\land \frac{t}{|y-x|^{d+\alpha}}\right), \quad x, y \in \R^d \setminus \left\{0\right\}, \ t>0,
\end{align}
and
\begin{align}\label{eq:tilde_pa_invariant}
\int_{\R^d} \tpa(t,x,y) |y|^{-\delta} dy = |x|^{-\delta},  \quad x \in \R^d \setminus \left\{0\right\}, \ t>0.
\end{align}
By \eqref{eq:domination_free_kernels} and Corollary \ref{cor:tildepest} applied to $k^{(1)}(t,x,y) = p(t,x,y)$, $k^{(2)}(t,x,y) = p^{(\alpha)}(t,x,y)$ and $\lambda = |\sigma|$, we get
\begin{align} \label{eq:est_by_stable_hardy}
\tp(t,x,y) \le e^{|\sigma| t} \tpa(t,x,y), \qquad x,y \in \R^d, \ t>0.
\end{align}
This domination property will be crucial for our further investigations in this paper. By using \eqref{eq:est_by_stable_hardy}, \eqref{eq:tilde_pa_est} and by following the argument in the proof of \cite[Lemmas 4.9-4.10]{KB-TG-TJ-DP-2019} (for any fixed $t>0$), we can show directly that the function 
\begin{align} \label{eq:cont_kernel}
(0,\infty) \times \left(\R^d \setminus \left\{0 \right\}\right)^2 \ni (t,x,y) \longmapsto \tp(t,x,y) \ \ \text{is continuous.}
\end{align} 

\subsection{Semigroup of operators defined by the kernel $\tp$} 
We define
$$
\tPa_t f(x) = \int_{\R^d} f(y) \tpa(t,x,y) dy, \quad f \in L^2(\R^d,dx), \ t>0,
$$ 
$$
\tP_t f(x) = \int_{\R^d} f(y) \tp(t,x,y) dy, \quad f \in L^2(\R^d,dx), \ t>0.
$$ 
It is shown in \cite[Proposition 2.4]{KB-TG-TJ-DP-2019} that $(\tPa_t)_{t > 0}$ is a strongly continuous semigroup of contractions on $L^2(\R^d,dx)$. By the same argument and the domination property \eqref{eq:est_by_stable_hardy} this extends to $(\tP_t)_{t > 0}$ except the $L^2$-contractivity. Indeed, we only know that $\big\|\tP_t\big\|_{L^2 \to L^2} \leq e^{|\sigma| t}$, $t >0$. Clearly, each $\tPa_t$ and $\tP_t$ is a self-adjoint operator. 

We will also need the following smoothing properties of the semigroup $(\tP_t)_{t > 0}$.

\begin{lemma} \label{eq:smooth}
Under \eqref{A1}, for every $t>0$, we have
\[
\tP_t \big(L^2(\R^d)\big) \subset C\big(\R^d \setminus \left\{0\right\}\big).
\]
Furthermore, there exists a constant $c>0$ such that, for every $x \in \R^d$, $t>0$ and $f \in L^2(\R^d)$,
\[
|\tP_t f(x)| \leq c e^{|\sigma| t} t^{-\frac{d-\delta}{\alpha}} \big(1+ t^{\frac{d}{2\alpha}}\big) \left(1+\frac{t^{\delta/\alpha}}{|x|^{\delta}}\right) \left\|f\right\|_2,
\]
where $\delta$ is defined by \eqref{eq:kappa_and_delta}. The constant $c$ does not depend on $\sigma$.
\end{lemma}

\begin{proof}
We first establish the continuity assertion. Fix $f \in L^2(\R^d)$ and $t>0$. Let $x, z \in \R^d \setminus \left\{0\right\}$ be such that $|x-z|< 1 \wedge |x|/2$ and let $R > |x| +1$. We have
\[
|\tP_t f(x) - \tP_t f(z)| \leq \int_{|y|\leq 2R} |f(y)| |\tp(t,x,y) - \tp(t,z,y)| \d{y} + \int_{|y| > 2R}|f(y)| (\tp(t,x,y) + \tp(t,z,y)) \d{y}.
\]
By \eqref{eq:est_by_stable_hardy} and \eqref{eq:tilde_pa_est}, we get
\begin{align} \label{eq:cont_aux}
|\tp(t,x,y) - \tp(t,z,y)| & \leq \tp(t,x,y) + \tp(t,z,y) \\
& \leq c_1 e^{|\sigma| t} \left(1+\frac{2t^{\delta/\alpha}}{|x|^{\delta}}\right) \left(\frac{2}{t^{d/\alpha}} \land \left(\frac{t}{|y-z|^{d+\alpha}}+\frac{t}{|y-x|^{d+\alpha}}\right)\right)\left(1+\frac{t^{\delta/\alpha}}{|y|^{\delta}}\right). \nonumber
\end{align}
Recall that $0 \leq \delta \leq(d-\alpha)/2$. Together with the Cauchy--Schwarz inequality this implies that $\int_{|y|\leq 2R} |f(y)| \left(1+\frac{t^{\delta/\alpha}}{|y|^{\delta}}\right)  \d{y} < \infty$. Consequently, the first integral on the right hand side above goes to zero as $z \to x$, by \eqref{eq:cont_kernel} and the Lebesgue dominated convergence theorem. 

On the other hand, by \eqref{eq:cont_aux} and the Cauchy--Schwarz inequality,  
\begin{align*}
\int_{|y| > 2R}|f(y)| & (\tp(t,x,y) + \tp(t,z,y)) \d{y} \\ 
                      & \leq c_1 e^{|\sigma| t} \left(1+\frac{2t^{\delta/\alpha}}{|x|^{\delta}}\right) \left(1+t^{\delta/\alpha}\right) \int_{|y| > 2R} |f(y)| \left(\frac{t}{|y-z|^{d+\alpha}}+\frac{t}{|y-x|^{d+\alpha}}\right)\d{y}\\
										  & \leq 2 c_1 t e^{|\sigma| t} \left(1+\frac{2t^{\delta/\alpha}}{|x|^{\delta}}\right) \left(1+t^{\delta/\alpha}\right)  \left\|f\right\|_2 \left(\int_{|y| > R}\frac{1}{|y|^{2d+2\alpha}}\d{y}\right)^{1/2}\\
											& \leq c_2 t e^{|\sigma| t} \left(1+\frac{2t^{\delta/\alpha}}{|x|^{\delta}}\right) \left(1+t^{\delta/\alpha}\right)  \left\|f\right\|_2 R^{-d/2 -\alpha}.
\end{align*}
Hence, 
\begin{align*}
\limsup_{z \to x} |\tP_t f(x) - \tP_t f(z)| & \leq c_2 t e^{|\sigma| t} \left(1+\frac{2t^{\delta/\alpha}}{|x|^{\delta}}\right) \left(1+t^{\delta/\alpha}\right)  \left\|f\right\|_2 R^{-d/2 -\alpha}
\end{align*}
and, by letting $R \to \infty$, we get that $\tP_t f(z) \to \tP_t f(x)$ as $z \to x$. Since $x \neq 0$ was arbitrary, this shows that $\tP_t f \in C\big(\R^d \setminus \left\{0\right\}\big)$. 

In order to get the second assertion, we proceed in a similar way. Let $f \in L^2(\R^d)$. By \eqref{eq:est_by_stable_hardy}, \eqref{eq:tilde_pa_est} and \eqref{eq:oppt}, we can write
\begin{align*}
|\tP_t f(x)| & \leq c_3 e^{|\sigma| t} \left(1+\frac{t^{\delta/\alpha}}{|x|^{\delta}}\right) \left(\int_{|y| \leq 1} + \int_{|y| > 1}\right) |f(y)| \pa(t,x,y) \left(1+\frac{t^{\delta/\alpha}}{|y|^{\delta}}\right)\d{y} \\
             & \leq c_3 e^{|\sigma| t} \left(1+\frac{t^{\delta/\alpha}}{|x|^{\delta}}\right) \left((1+t^{\delta/\alpha})\int_{\R^d} |f(y)| \pa(t,x,y) \d{y} + t^{-(d-\delta)/\alpha}\int_{|y| \leq 1}\frac{|f(y)|}{|y|^{\delta}}\d{y}\right) 
\end{align*}
and, by the Cauchy--Schwarz inequality applied to both integrals and \eqref{eq:oppt}, we get 
\begin{align*}
|\tP_t f(x)| & \leq c_3 e^{|\sigma| t} \left(1+\frac{t^{\delta/\alpha}}{|x|^{\delta}}\right) \left((1+t^{\delta/\alpha})\left(\int_{\R^d} |f(y)|^2 \pa(t,x,y) \d{y}\right)^{1/2} + c_4 t^{-(d-\delta)/\alpha} \left\|f\right\|_2 \right)  \\
             & \leq c_5 e^{|\sigma| t} \left(1+\frac{t^{\delta/\alpha}}{|x|^{\delta}}\right) \left((1+t^{\delta/\alpha}) t^{-d/(2\alpha)} + t^{-(d-\delta)/\alpha} \right) \left\|f\right\|_2 \\
						 & \leq c_6 e^{|\sigma| t} \left(1+\frac{t^{\delta/\alpha}}{|x|^{\delta}}\right)  t^{-(d-\delta)/\alpha}(1+t^{d/(2\alpha)}) \left\|f\right\|_2.
\end{align*}
This completes the proof.
\end{proof}
Observe that the above lemma also applies directly to the semigroup $(\tPa_t)_{t > 0}$ as a special case (we just take $\sigma \equiv 0$ in \eqref{A1}).

\section{Schr\"odinger operators with singular potentials and their semigroups} \label{sec:q_forms}

\subsection{Quadratic forms of free operators}
We first discuss the relation between quadratic forms of the operators $L^{(\alpha)}$ and $L$. Due to \eqref{eq:relation_kernels} our analysis will be based on the corresponding operator semigroups.
For $f \in L^2(\R^d)$, we define
\begin{align*}
\E_t[f] = \frac{1}{t}\spr{f-P_tf,f}, \qquad \E[f] = \lim_{t \searrow 0}\E_t[f].
\end{align*} 
Since the operators $P_t$ are self-adjoint contractions on $L^2(\R^d)$, it follows from spectral theorem that the map
\[
(0,\infty) \ni t \mapsto \E_t[f]
\] 
is decreasing for any $f \in L^2(\R^d)$, see e.g.\ \cite[Lemma 1.3.4]{Fuk}. In particular, the above limit exists and belongs to $[0,\infty]$. 
The domain of $\E$ is defined as
\begin{align*}
\D(\E) = \{f \in L^2 \colon \E[f] <\infty\}.
\end{align*} 
For $f,g \in L^2(\R^d)$ we also define
\begin{align*}
\E_t[f,g] = \frac{1}{t}\spr{f-P_tf,g}, \quad t>0.
\end{align*} 
It is easy to check that
\[
\E_t[f,g] = \frac{1}{2}\left(\E_t[f+g]-\E_t[f]-\E_t[g]\right),
\]
which shows that the limit 
\begin{align*}
\E[f,g] = \lim_{t \searrow 0}\E_t[f,g]
\end{align*} 
exists and is finite for any $f,g \in \D(\E)$; $\E[f,g]$ defines a bilinear form corresponding to initial quadratic form. Consequently, the following polarization formula
\[
\E[f,g] = \frac{1}{2}\left(\E[f+g]-\E[f]-\E[g]\right),
\]
holds for $f,g \in \D(\E)$.

The form $\big(\Ea,\D(\Ea)\big)$ of the fractional Laplacian $L^{(\alpha)}$ is defined in the same manner by replacing the operators $P_t$ with $\Pa_t$. Due to \eqref{A1}, it is a special case of $\big(\E,\D(\E)\big)$.  

Below we need to consider the convolution operators
\[
\sigma^{k*} f(x) = \int_{\R^d} \sigma^{k*}(x-y) f(y) dy, \quad k \in \N, \ f \in L^2(\R^d),
\]
where $\sigma$ comes from \eqref{A1}. Since $\sigma^{k*} \in L^1(\R^d)$, $k \in \N$, it defines a bounded operator on $L^2(\R^d)$. Indeed, by the Cauchy--Schwarz inequality and the Tonelli theorem, for $f \in L^2(\R^d)$, we have
\begin{align} \label{eq:sigma_norm}
\left\|\sigma^{k*} f\right\|_2^2 & \leq \int_{\R^d} \left(\int_{\R^d}\sigma^{k*}(x-y)dy\right)  \left(\int_{\R^d} \sigma^{k*}(x-y) |f(y)|^2 dy\right) dx \nonumber \\ & = |\sigma|^k  \int_{\R^d} \left(\int_{\R^d}\sigma^{k*}(x-y)dx\right) |f(y)|^2 dy = |\sigma|^{2k} \left\|f\right\|_2^2. 
\end{align}
\begin{lemma}\label{lem:domainsEaE}
Under assumption \eqref{A1}, we have
\begin{align*}
\E[f] + |\sigma|\|f\|_2^2 - \spr{\sigma f, f} = \Ea[f], \quad f \in L^2(\R^d).
\end{align*}
In particular, $\D(\E) = \D(\Ea)$,
\begin{align*}
\D(\E) = \left\{f \in L^2(\R^d): \int_{\R^d}\int_{\R^d} \big(f(x) - f(y)\big)^2 \nu(x-y) \d{x}\d{y} < \infty \right\}
\end{align*}
and 
\begin{align*}
\E[f] = \frac{1}{2} \int_{\R^d}\int_{\R^d} \big(f(x) - f(y)\big)^2 \nu(x-y) \d{x}\d{y}, \quad f \in L^2(\R^d).
\end{align*}
\end{lemma}
\begin{proof}
We first establish the first identity. For every $t>0$, by \eqref{eq:relation_kernels} we have
\begin{align*} 
\spr{f-\Pa_t f,f} & = \spr{f - e^{-|\sigma|t}f,f} + e^{-|\sigma|t}\spr{f-P_t f,f} \\
& \ \ \ - \spr{ e^{-|\sigma|t} t P_t(\sigma f),f} - \spr{ e^{-|\sigma|t} \sum_{k=2}^\infty \frac{t^k}{k!}P_t(\sigma^{k*}f),f}. 
\end{align*}
Hence, 
\begin{align}\label{eq:approx_free_form}
\Ea_t[f] & = \frac{1}{t}\spr{f - e^{-|\sigma|t}f,f} + e^{-|\sigma|t}\E_t[f] \nonumber\\
& \ \ \ -  e^{-|\sigma|t}\spr{ P_t(\sigma f),f} - e^{-|\sigma|t} \spr{  \sum_{k=2}^\infty \frac{t^{k-1}}{k!}P_t(\sigma^{k*}f),f}. 
\end{align}
As shown above, $\sigma f \in L^2(\R^d)$. Thus, by strong continuity, $\spr{ P_t(\sigma f),f} \to \spr{\sigma f,f}$, as $t \searrow 0$. Moreover, by \eqref{eq:sigma_norm},
\begin{align*}
\left|\spr{\sum_{k=2}^\infty \frac{t^{k-1}}{k!}P_t(\sigma^{k*}f),f}\right| 
         & \leq \sum_{k=2}^\infty \frac{t^{k-1}}{k!} \left|\spr{P_t(\sigma^{k*}f),f}\right| \\ 
				  & \leq \sum_{k=2}^\infty \frac{t^{k-1}}{k!} \left\|P_t(\sigma^{k*}f)\right\|_2 \left\|f\right\|_2 \\ 
				 & \leq \sum_{k=2}^\infty \frac{t^{k-1}}{k!} |\sigma|^k \left\|f\right\|_2^2 \leq t |\sigma|^2 \left\|f\right\|_2^2 e^{|\sigma|t} \to 0, \quad \text{as} \ t \searrow 0. 
\end{align*}
Then, by letting $t \searrow 0$ in \eqref{eq:approx_free_form}, we get the first identity of the lemma. In particular, we see that $\D(\E) = \D(\Ea)$. 

In order to see the second assertion, we observe that
\begin{align*}
|\sigma|\|f\|_2^2 - \spr{\sigma f, f} & = \int_{\R^d} f(x) \int_{\R^d} (f(x)-f(y)) \sigma(x-y)dy dx \\
                                      & = \frac{1}{2} \int_{\R^d} \int_{\R^d} (f(x)-f(y))^2 \sigma(x-y)dy dx \geq 0, \quad f \in L^2(\R^d),
\end{align*}
by symmetrization. By \eqref{eq:sigma_norm} this double integral is finite for every $f \in L^2(\R^d)$. Furthermore, it is known that
\begin{align*}
\Ea[f] = \frac{1}{2} \int_{\R^d}\int_{\R^d} \big(f(x) - f(y)\big)^2 \nua(x-y) \d{x}\d{y}, \quad f \in L^2(\R^d).
\end{align*}
Therefore, by the first identity of the lemma proved above, we see that 
\[
\E[f]<\infty \quad \Longleftrightarrow \quad \int_{\R^d}\int_{\R^d} \big(f(x) - f(y)\big)^2 \nua(x-y) \d{x}\d{y} < \infty.
\]
Consequently, by \eqref{A1},
\[
\E[f]<\infty \quad \Longleftrightarrow \quad \int_{\R^d}\int_{\R^d} \big(f(x) - f(y)\big)^2 \nu(x-y) \d{x}\d{y} < \infty
\]
and the equality 
\[
\E[f] = \frac{1}{2} \int_{\R^d}\int_{\R^d} \big(f(x) - f(y)\big)^2 \nua(x-y) \d{x}\d{y} - \frac{1}{2} \int_{\R^d} \int_{\R^d} (f(x)-f(y))^2 \sigma(x-y)dy dx
\]
holds for $f \in L^2(\R^d)$. This proves the second assertion of the lemma.
\end{proof}

\subsection{Quadratic form of the semigroup $(\tP_t)_{t > 0}$}

In this section we identify the form $\tE$ of the semigroup $(\tP_t)_{t > 0}$ when acting on $\D(\E)$. 
Similarly as above, we define for $f \in L^2(\R^d)$
\begin{align*}
\tE_t[f] = \frac{1}{t}\spr{f-\tP_tf,f}, \qquad \tE[f] = \lim_{t\to0}\tE_t[f] .
\end{align*} 
The domain of $\tE$ is
\begin{align*}
\D(\tE) = \{f \in L^2 \colon \tE[f] <\infty\},
\end{align*} 
see \cite[Chapter 6]{vanCasteren} and \cite[Chapter 10]{Schmudgen}.
The corresponding bilinear form is defined as
\begin{align*}
\tE[f,g] =\lim_{t \to 0} \frac{1}{t}\spr{f-\tP_tf,g},
\end{align*} 
for any $f,g \in \D(\tE)$, and the following polarization formula 
\[
\tE[f,g] = \frac{1}{2}\left(\tE[f+g]-\tE[f]-\tE[g]\right),
\]
holds for $f,g \in \D(\tE)$.
Clearly, the form $\big(\tEa,\D(\tEa)\big)$ associated with the semigroup $(\tPa_t)_{t > 0}$ is a special case of $\big(\tE,\D(\tE)\big)$.
The following lemma partly extends \cite[Lemma 5.1]{KB-TG-TJ-DP-2019}. 

\begin{lemma}\label{lem:domainsEtE}
Under \eqref{A1}, for $\kappa \leq \kappa^*$, we have $\D(\E) \subset \D(\tE)$ and
\begin{align}\label{Eq:domainsEtE}
\tE[f] &= \E[f] - \int_{\R^d} V(x)f^2(x)  \d{x}, \quad f \in \D(\E).
\end{align}

\end{lemma}
\begin{proof}
Let $0 \le f \in \D(\E)$. By perturbation formula \eqref{eq:perturbation} and Tonelli's theorem,
\begin{align*}
\tE_t[f] & = \E_t[f] - \int_{\Rd}\int_{\Rd} f(x) f(y) \frac{1}{t} \int_0^t \int_{\Rd} p(u,x,z) V(z)\tp(t-u,z,y) \d{z}\d{u}\d{x}\d{y} \\
         & = \E_t[f] - \int_{\Rd} V(z) \frac{1}{t} \int_0^t P_uf(z)  \tP_{t-u}f(z) \d{u} \d{z} .
\end{align*}
First note that 
\begin{align}\label{eq:L1conv}
\frac{1}{t} \int_0^t P_uf  \tP_{t-u}f \d{u} \stackrel{L^1}{\longrightarrow} f^2,
\end{align}
as $t \to 0$. Indeed,
\begin{align*}
P_uf  \tP_{t-u}f - f^2 = (P_uf -f)\tP_{t-u}f + f(\tP_{t-u}f -f).
\end{align*}
Hence, by H\"older inequality,
\begin{align*}
\left\|\frac{1}{t} \int_0^t P_uf  \tP_{t-u}f \d{u} - f^2  \right\|_1 &= \left\|\frac{1}{t} \int_0^t (P_uf  \tP_{t-u}f  - f^2)  \d{u}\right\|_1 \\
&\le \frac{1}{t} \int_0^t \left\|(P_uf - f) \tP_{t-u}f + f(\tP_{t-u}f  - f) \right\|_1 \d{u}\\
&\le \frac{1}{t} \int_0^t \big(\norm{P_uf - f}{2} \norm{\tP_{t-u}f}{2} + \norm{f}{2} \norm{\tP_{t-u}f  - f}{2}\big) \d{u}.
\end{align*}
Recall that $\{P_t\}_{t>0}$ and $\{\tP_t\}_{t>0}$ are strongly continuous semigroups of bounded operators on $L^2(\Rd)$. Therefore, 
\begin{align*}
\left\|\frac{1}{t} \int_0^t P_uf  \tP_{t-u}f \d{u} - f^2  \right\|_1 
&\le \norm{f}{2} \left(\frac{1}{t}\int_0^t \norm{P_uf - f}{2} e^{|\sigma|(t-u)} + \norm{\tP_{t-u}f  - f}{2} \d{u}\right) \\
&\le \norm{f}{2} e^{|\sigma|t} \sup_{u \in (0,t]}  \left(\norm{P_uf - f}{2} + \norm{\tP_{t-u}f  - f}{2} \right) \stackrel{t \to 0}{\longrightarrow} 0.\\
\end{align*}
Now, let $t_n \to 0$ be a sequence  such that 
\begin{align*}
\liminf_{t \to 0}  \int_{\Rd} V(z) \frac{1}{t} \int_0^t P_uf(z)  \tP_{t-u}f(z) \d{u} \d{z} = \lim_{n \to \infty}   \int_{\Rd} V(z) \frac{1}{t_n} \int_0^{t_n} P_uf(z) \tP_{t_n-u}f(z) \d{u}\d{z}
\end{align*}
By \eqref{eq:L1conv}, we can choose a subsequence $t_{n_k}\to 0$ such that
\begin{align*}
\frac{1}{t_{n_k}} \int_0^{t_{n_k}} P_uf(z) \tP_{t_{n_k}-u}f(z) \d{u} \stackrel{k\to\infty}{\longrightarrow} f^2(z), \quad \text{a.e.} \ z \in \Rd.
\end{align*}
Hence, by Fatou's lemma, we have
\begin{align*}
\limsup_{t \to 0}\tE_{t}(f) &\le \E[f] - \liminf_{t \to 0}  \int_{\Rd} V(z) \frac{1}{t} \int_0^t P_uf(z)  \tP_{t-u}f(z) \d{u} \d{z}\\ 
&= \E[f] - \lim_{k \to \infty}   \int_{\Rd} V(z) \frac{1}{t_{n_k}} \int_0^{t_{n_k}} P_uf(z) \tP_{t_{n_k}-u}f(z) \d{u} \\
&\le \E[f] -  \int_{\Rd} V(z)  \lim_{k \to \infty}  \frac{1}{t_{n_k}}\int_0^{t_{n_k}}  P_uf(z) \tP_{t_{n_k}-u}f(z) \d{u}\d{z} \\
&= \E[f] -  \int_{\Rd}  V(z) f^2(z) \d{z}.
\end{align*}
On the other hand, since $0 \le f \in \D(\E) = \D(\E^{(\alpha)})$, by \eqref{eq:domination_free_kernels} and \eqref{eq:est_by_stable_hardy},
\begin{align*}
\liminf_{t \to 0}\tE_{t}[f] & \ge \E[f] - \limsup_{t \to 0}  \int_{\Rd}  V(z) \frac{1}{t} \int_0^{t} P_uf(z) \tP_{t-u}f(z) \d{u}\d{z}\\ 
& \ge \E[f] - \limsup_{t \to 0}   e^{|\sigma| t} \int_{\Rd}  V(z) \frac{1}{t} \int_0^{t} P_u^{(\alpha)}f(z)\tP^{(\alpha)}_{t-u}f(z) \d{u}\d{z},
\end{align*}
and by using \eqref{eq:perturbation} for $\pa(t,x,y)$ and $\tpa(t,x,y)$, we get 
\[
\int_{\Rd}  V(z) \frac{1}{t} \int_0^{t} P_u^{(\alpha)}f(z)\tP^{(\alpha)}_{t-u}f(z) \d{u}\d{z} = \Ea_t[f] - \tEa_t[f]. 
\]
Consequently, by \cite[Lemma 5.1]{KB-TG-TJ-DP-2019}, 
\[
\lim_{t \to 0} \int_{\Rd}  V(z) \frac{1}{t} \int_0^{t} P_u^{(\alpha)}f(z)\tP^{(\alpha)}_{t-u}f(z) \d{u}\d{z} = \int_{\R^d} V(z)f^2(z)dz.
\]
Therefore $f \in \D(\tE)$ and \eqref{Eq:domainsEtE} holds for $0 \le f \in \D(\E)$. 
Now, consider arbitrary $f \in \D(\E)$. We observe that $|f| \in \D(\E)$. Indeed, 
\begin{align*}
\E_t[f,f] = \frac{1}{t}\left( \langle f,f\rangle - \left\langle P_{t/2}f, P_{t/2}f\right\rangle\right) \ge \frac{1}{t}\left( \langle |f|,|f|\rangle - \left\langle P_{t/2}|f|, P_{t/2}|f|\right\rangle\right) = \E_t[|f|,|f|].
\end{align*}
Hence, $\E(f,f) \ge \E(|f|,|f|)$, which yields $|f| \in \D(\E)$. Finally, $f =f_+ - f_- \in \D(\tE)$, since $f_+ = \frac{|f|+f}{2} \in \D(\tE)$, $f_- = \frac{|f|-f}{2} \in \D(\tE)$ and $\D(\tE)$ is a linear space.

We are left to show that \eqref{Eq:domainsEtE} extends to an arbitrary (in particular, signed) $f \in \D(\E)$.
But this holds by standard polarization identities for the forms $\E$ and $\tE$, see e.g.\ the last lines of the proof of \cite[Lemma 5.1]{KB-TG-TJ-DP-2019}.

\end{proof}

\subsection{Schr\"odinger operators with subcritical singular potential and their semigroups} \label{sec:SchOp_def}
Throughout this section we assume that $\kappa < \kappa^*$, i.e.\ we consider only the subcritical potential $V(x) = \kappa|x|^{-\alpha}$. In this case the Schr\"odinger operator $H= -L-V$ can be defined as a form-sum. Indeed, by the fractional Hardy inequality \eqref{eq:frac_Hardy} and the identity in Lemma \ref{lem:domainsEaE}, we easily get the inequalities
\[
\kappa \int_{\Rd} \frac{f^2(x)}{|x|^{\alpha}} \d{x} \leq \frac{\kappa}{\kappa^*} \Ea[f] \leq \frac{\kappa}{\kappa^*} \E[f] + \frac{\kappa}{\kappa^*} |\sigma| \left\|f\right\|^2_2, \quad f \in \D\big(\E\big).
\]
This means that the form of the potential $V(x) = \kappa|x|^{-\alpha}$ is relatively $\E$-bounded with relative $\E$-bound $\kappa/\kappa^* < 1$. Hence, by the \emph{KLMN theorem} (see e.g.\ \cite[Theorem 10.21]{Schmudgen}), there exists a unique, bounded below, self-adjont operator $H$, called the form sum of the operators $-L$ and $-V$ (we simply write $H=-L-V$), such that the form $\big(\E_H, \D(\E_H)\big)$ of $H$ satisfies 
\[
\D(\E_H) = \D(\E) \quad \text{and} \quad \E_H[f,g] = \E[f,g] - \int_{\R^d}V(x) f(x)g(x) \d{x}, \quad f,g \in \D(\E_H).
\]

We will now show that the form of the Schr\"odinger operator $H$ and the form $\tE$ discussed in the previous section are equal. Consequently, the strongly continuous semigroup of operators corresponding to $H$ can be identified with the semigroup $\big\{\widetilde{P}_t: t \geq 0\big\}$ constructed above. In particular, $e^{-tH}$, $t>0$, are integral operators with kernels $\tp(t,x,y)$.
This is crucial for our further investigations.

\begin{theorem} \label{th:forms_equal} Under \eqref{A1}, for $\kappa< \kappa^*$, we have
\begin{align} \label{eq:equality_of_domains}
\D\big(\tE\big) = \D\big(\tEa\big)  = \D\big(\Ea\big) =   \D(\E).
\end{align}
In particular, $\big(\E_H, \D(\E_H)\big)  = \big(\tE, \D\big(\tE\big)\big)$ and the semigroups $\big\{e^{-tH}:t \geq 0\big\}$ and $\big\{\widetilde{P}_t: t \geq 0\big\}$ are equal on $L^2(\Rd)$.
\end{theorem}
\begin{proof}
The equality $ \D\big(\Ea\big) =   \D(\E)$ and the inclusion  $\D(\E) \subset \D\big(\tE\big)$ were established in Lemmas \ref{lem:domainsEaE} and \ref{lem:domainsEtE}, respectively. In order to complete the proof of \eqref{eq:equality_of_domains} we need to show the inclusions $\D\big(\tE\big) \subset \D\big(\tEa\big)  \subset \D\big(\Ea\big)$. Observe that by \eqref{eq:est_by_stable_hardy} we have for $0 \leq f \in L^2(\R^d)$
\begin{align*}
\tE_t[f] = \frac{1}{t}\left(\spr{f,f}-\spr{\tP_tf,f}\right) 
      & \geq \frac{1}{t}\left(\spr{f,f}-e^{|\sigma t|}\spr{\tPa_tf,f}\right) \\
			& = \frac{1}{t}\left(\spr{f,f}-\spr{\tPa_tf,f}\right) 
			  - \frac{1}{t} \left(e^{|\sigma t|}-1\right) \left\|\tPa_{\frac{t}{2}}f\right\|_2^2 \\ 
			&	\geq \tEa_t[f] - |\sigma| e^{|\sigma t|} \left\|f\right\|_2^2.
\end{align*}
Letting $t \to 0$, we see that if $0 \leq f \in \D\big(\tE\big)$, then $f \in \D\big(\tEa\big)$. By the same argument as in the proof of Lemma \ref{lem:domainsEtE}, this extends to a general $f = f_+-f_- \in \D\big(\tE\big)$, giving the full inclusion $\D\big(\tE\big) \subset \D\big(\tEa\big)$.

We are left to show that $\D\big(\tEa\big)  \subset \D\big(\Ea\big)$. First note that by \cite[Lemma 5.1]{KB-TG-TJ-DP-2019} and the fractional Hardy inequality \eqref{eq:frac_Hardy}, we have $\D\big(\Ea\big) \subset \D\big(\tEa\big)$ and
\begin{align*}
\Ea[f] = \tEa[f] + \frac{\kappa}{\kappa^*} \kappa^*\int_{\Rd} \frac{f^2(x)}{|x|^{\alpha}}  \d{x} \le \tEa[f] + \frac{\kappa}{\kappa^*}\Ea[f],
\quad f \in \D\big(\Ea\big).
\end{align*}
Hence,
\begin{align}\label{ineq:tEaEa}
\Ea[f] \le  \frac{\kappa^*}{\kappa^*-\kappa}\tEa[f], \quad f \in \D\big(\Ea\big).
\end{align}
By \cite[Theorem 5.4]{KB-TG-TJ-DP-2019}, $C_c^\infty(\Rd)$ is dense in $\D\big(\tEa\big)$ with the norm $\sqrt{\tEa[\cdot]} + \|\cdot\|_{2}$. Furthermore, $C_c^\infty(\Rd) \subset \D\big(\Ea\big)$. We will now use these facts to complete the proof.

Let $f \in \D\big(\tEa\big)$ and let $f_n \in C_c^\infty(\Rd)$ be a sequence such that $\tEa[f-f_n] \to 0$ and $\|f-f_n\|_{2} \to 0$ as $n \to \infty$. In particular, $\tEa[f_n-f_m] \to 0$ and $\|f_n-f_m\|_{2} \to 0$ as $n, m \to \infty$. By \eqref{ineq:tEaEa} we obtain that $\Ea[f_n-f_m] \to 0$ as $n, m \to \infty$. Since $\big(\Ea,\D\big(\Ea\big)\big)$ is a closed form, we obtain that $f \in \D\big(\Ea\big)$, showing the inclusion $\D\big(\tEa\big) \subset \D\big(\Ea\big)$. 

The equality  $\big(\E_H, \D(\E_H)\big)  = \big(\tE, \D\big(\tE\big)\big)$ follows directly from \eqref{eq:equality_of_domains} and Lemma \ref{lem:domainsEtE}. Since, by definition, both these forms are symmetric, bounded below and closed bilinear forms, they uniquely determine the strongly continuous semigroups of bounded self-adjoint operators on $L^2(\R^d)$, see e.g.\ \cite[Theorem 6.2 (b)]{vanCasteren}.
This completes the proof.
\end{proof}

\section{Pointwise estimates of eigenfunctions} \label{sec:ef_est}

In this section we find pointwise estimates for eigenfunctions of the Schr\"odinger operators $H = -L-V$ with subcritical potential $V(x)=\kappa |x|^{-\alpha}$, i.e.\ we consider the case $\kappa < \kappa^*$.

Throughout this section we assume that $\varphi \in L^2(\R^d)$ is an eigenfunction of the operator $H$ corresponding to a negative eigenvalue, i.e.
\begin{align} \label{eq:eigenfunction}
\text{there exists} \quad E<0 \quad \text{such that} \quad H\varphi = E \varphi.  
\end{align}
It is now crucial that by Theorem \ref{th:forms_equal} we have
\begin{align} \label{eq:eigenfunction_representation}
\varphi(x) = e^{Et} e^{-tH} \varphi(x) = e^{Et} \tP_t \varphi(x) = e^{Et} \int_{\R^d} \tp(t,x,y)\varphi(y) \d{y}, \quad t>0.
\end{align}

We first establish the continuity of eigenfunctions and the upper estimate around zero. 
We always assume that any eigenfunction $\varphi$ is normalized so that $\left\|\varphi\right\|_2 = 1$. 

\begin{proposition} \label{lem:est_at_zero}
Let \eqref{A1} hold and let $\kappa< \kappa^*$. Any $\varphi \in L^2(\R^d)$ satisfying \eqref{eq:eigenfunction} has a version which is continuous on $\R^d \setminus \left\{0\right\}$ and satisfies the following upper estimate: there exists $c>0$ such that
\[
|\varphi(x)| \leq c e^{|\sigma|+E} \left(1+\frac{1}{|x|^{\delta}}\right), \quad x \in \R^d,
\]
where $\delta$ is determined by \eqref{eq:kappa_and_delta}. The constant $c$ depends neither on $\varphi$, $E$ nor $\sigma$. 
\end{proposition}

\begin{proof}
It follows directly from \eqref{eq:eigenfunction_representation} and Lemma \ref{eq:smooth} by taking $t=1$.
\end{proof}
In what follows, we work with the version of the eigenfunction $\varphi$ which is continuous on $\R^d \setminus \left\{0\right\}$. In particular, the eigenequations \eqref{eq:eigenfunction_representation} can always be understood pointwise.

Estimates at infinity will be given in terms of the resolvent kernel of the free operator $L$: 
\[
g_\lambda(x) = \int_0^\infty e^{-\lambda t} p(t,x) \d{t},  \qquad x \in \Rd, \ \lambda>0.
\]
It is not difficult to show that for every $\lambda>0$ the map $x \mapsto g_{\lambda}(x)$ is continuous on $\Rd \setminus \left\{0\right\}$. For a Borel function $f \geq 0$ we also define
\[
G_\lambda f(x) = \int_0^\infty e^{-\lambda t} P_t f(x) \d{t} = \int_{\Rd} g_\lambda(x-y) f(y) \d{y}, \quad \lambda>0.
\]
The last equality is a consequence of the Tonelli theorem. Clearly, the definition easily extends to any (signed) Borel $f$ 
for which the integrals are absolutely convergent. 

\begin{lemma}\label{lem:eigf_est}
Let \eqref{A1} hold and let $\kappa< \kappa^*$. If $\varphi \in L^2(\R^d)$ is such that \eqref{eq:eigenfunction} holds,
then
\begin{align}\label{Eq1:eigf_est}
|\varphi(x)| \le G_{|E|}(V|\varphi|)(x), \qquad x \in \Rd \setminus \left\{0\right\}.
\end{align}
If, in addition, $\varphi \ge 0$, then 
\begin{align}\label{Eq2:eigf_est}
\varphi(x) = G_{|E|} (V\varphi)(x), \qquad x \in \Rd \setminus \left\{0\right\}.
\end{align} 
\end{lemma}
\begin{proof}
We first show \eqref{Eq1:eigf_est}. Recall that $E<0$. We let $\lambda:=|E|=-E >0$ to simplify the notation. By the eigenequation \eqref{eq:eigenfunction_representation} and the perturbation formula \eqref{eq:perturbation}, we have for $x \in \Rd \setminus \left\{0\right\}$
\begin{align}\label{eq:phipert}
\varphi(x) = e^{-\lambda t} \tP_t \varphi(x) 
&= e^{-\lambda t} P_t \varphi(x) + \int_0^t \int_{\Rd} e^{-\lambda s} p(s,x,z)V(z) e^{-\lambda(t-s)} \tP_{t-s} \varphi(z) \d{z}\d{s} \notag \\
&= e^{-\lambda t} P_t \varphi(x) + \int_0^t \int_{\Rd} e^{-\lambda s} p(s,x,z)V(z) \varphi(z) \d{z}\d{s}.
\end{align}
The change of the order of integration here is possible due to Fubini's theorem as we have
\[
\int_{\R^d} \int_0^t \int_{\R^d} e^{-\lambda s} p(s,x,z) V(z) e^{-\lambda(t-s)} \tp(t-s,z,y)|\varphi(y)| \d z \d s \d y \leq e^{-\lambda t} \tP_t |\varphi|(x) < \infty,
\]
by \eqref{eq:perturbation} and Lemma \ref{eq:smooth}. Let $\eps \in (0,1)$. By \eqref{eq:phipert}, we have
\begin{align*}
e^{-\eps t} |\varphi(x)| \leq  e^{-(\lambda+\eps)t} P_t|\varphi|(x) + e^{-\eps t} \int_0^t e^{-\lambda s}P_s (V |\varphi|)(x) \d{s}
\end{align*}
and, by integrating with respect to $t \in (0,\infty)$ on both sides of this inequality, we get
\begin{align}\label{eq:phiequal}
|\varphi(x)| \leq \eps G_{\lambda + \eps} |\varphi|(x) +  G_{\lambda+\eps}(V|\varphi|)(x).
\end{align}
For the last term on the right hand side we used Tonelli's theorem. Since the map $\lambda \mapsto g_{\lambda}(x)$ is decreasing, \eqref{eq:phiequal} leads to the final estimate
\[
|\varphi(x)| \leq \eps G_{\lambda} |\varphi|(x) +  G_{\lambda}(V|\varphi|)(x).
\] 
We are left to make sure that 
\begin{align} \label{eq:final_need} 
G_{\lambda} |\varphi|(x) < \infty, \quad x \neq 0.
\end{align} 
Once we know this holds, we can get \eqref{Eq1:eigf_est} by letting $\eps \to 0$. To this end, we first use the Cauchy--Schwarz inequality to show that 
\[
P_t |\varphi|(x) \leq \left(\int_{\R^d} p(t,x,y)\d y\right)^{1/2} \left(\int_{\R^d} p(t,x,y) |\varphi(y)|^2 \d y \right)^{1/2} = \left(\int_{\R^d} p(t,x,y) |\varphi(y)|^2 \d y\right)^{1/2} .
\]
Observe now that for $t \geq 1$ we have
\[
  p(t,x,y) = p_t(y-x) \leq p_t(0) = (2\pi)^{-d}\int_{\R^d} e^{-t \psi(\xi)}\d \xi \leq (2\pi)^{-d}\int_{\R^d} e^{-\psi(\xi)}\d \xi < \infty, \quad  x, y \in \R^d,
\]
by \eqref{eq:free_density}. This implies that
\[
\int_{\R^d} p(t,x,y) |\varphi(y)|^2 \d y \leq \frac{\left\|\varphi\right\|_2^2}{(2\pi)^{d}}\int_{\R^d} e^{-\psi(\xi)}\d \xi, \quad t \geq 1.
\]
On the other hand, for $t \in (0,1)$ and $x \neq 0$, 
\begin{align*}
\int_{\R^d} p(t,x,y) |\varphi(y)|^2 \d y & = \int_{|y| \leq 1 \wedge |x|/2} p(t,x,y) |\varphi(y)|^2 \d y
                                                           + \int_{|y| > 1 \wedge |x|/2} p(t,x,y) |\varphi(y)|^2 \d y \\
																				 & \leq  \sup_{|y| \leq 1 \wedge |x|/2} p(t,x,y) \left\|\varphi\right\|^2_2 
                                                           + \sup_{|y| > 1 \wedge |x|/2} |\varphi(y)|^2.									
\end{align*}
It follows from \eqref{eq:domination_free_kernels} and \eqref{eq:oppt} that
\[
\sup_{|y| \leq 1 \wedge |x|/2} p(t,x,y) \leq c_1 e^{|\sigma|} (2/|x|)^{d+\alpha}, \quad t \in (0,1).
\]
Moreover, by \eqref{eq:eigenfunction_representation} and the estimate in Lemma \ref{eq:smooth},
\[
\sup_{|y| > 1 \wedge |x|/2} |\varphi(y)|^2 \leq e^{2\lambda} \sup_{|y| > 1 \wedge |x|/2} (\tP_1|\varphi(y)|)^2 \leq c_2 \big(1+1/(1 \wedge |x|/2)\big)^2 \left\|\varphi\right\|^2_2,
\]
respectively. By putting together all the estimates above, we conclude that $\sup_{t>0} P_t |\varphi|(x) < \infty$, for every $x \neq 0$. This clearly gives \eqref{eq:final_need} and completes the proof of \eqref{Eq1:eigf_est}.

In order to get the identity \eqref{Eq2:eigf_est}, we just come back to \eqref{eq:phipert} and observe that now we have an equality in \eqref{eq:phiequal} because $\varphi \geq 0$. In particular, $\varphi(x) \geq  G_{\lambda+\eps}(V \varphi)(x)$. Finally, by letting $\eps \to 0$, we get
\[
 \varphi(x) \geq  G_{\lambda}(V \varphi)(x).
\]
Together with \eqref{Eq1:eigf_est} proven above, this gives the assertion \eqref{Eq2:eigf_est}.
\end{proof}

Below we need the following identity. It seems to be a standard fact, but we provide here a short proof for reader's convenience. 
\begin{lemma}\label{lem:Glconv}
For $\lambda>0$ we have
\begin{align}\label{Eq:Glconv}
g_\lambda^{\ast n}(x) = \int_0^\infty \frac{t^{n-1}}{(n-1)!}e^{-\lambda t} p_t(x) \d{t}\, \qquad x\in\Rd, \; n\in\N.
\end{align}
\end{lemma}
\begin{proof}
We use induction. Clearly \eqref{Eq:Glconv} holds for $n=1$. Assume \eqref{Eq:Glconv} holds for some $n\in\N$. Then, by the Tonelli theorem and the Chapman--Kolmogorov identity \eqref{eq:ChKforp},
\begin{align*}
g_\lambda^{\ast (n+1)}(x) &= \int_{\R^d} g_\lambda^{\ast n}(x-y) g_\lambda(y)\d{y}\\
& = \int_{\R^d} \int_0^\infty \int_0^\infty \frac{t^{n-1}}{(n-1)!}e^{-\lambda t} p_t(x-y) e^{-\lambda s} p_s(y) \d{s} \d{t} \d{y}\\
& = \int_0^\infty \int_0^\infty \frac{t^{n-1}}{(n-1)!}e^{-\lambda (t+s)} p_{t+s}(x) \d{s} \d{t} \\
& = \int_0^\infty \int_t^\infty \frac{t^{n-1}}{(n-1)!}e^{-\lambda s} p_{s}(x) \d{s} \d{t} \\
& = \int_0^\infty \int_0^s \frac{t^{n-1}}{(n-1)!}e^{-\lambda s} p_{s}(x) \d{t} \d{s} 
= \int_0^\infty \frac{s^{n}}{n!}e^{-\lambda s} p_{s}(x)\d{s}.
\end{align*} 
\end{proof}

We are now in a position to make a concluding step in this section.

\begin{lemma}\label{lem:EFest}
Let \eqref{A1} hold and let $\kappa< \kappa^*$. If $\varphi \in L^2(\R^d)$ is such that \eqref{eq:eigenfunction} holds,
then for every $\eps \in (0,|E| \land 1)$ there is $R = R(\eps) \geq 1$ such that
\begin{align}\label{Eqcor:EFest}
	|\varphi(x)| \le \int_{|y| \le R} g_{|E|-\eps}(x-y) V(y) |\varphi(y)| \d{y}, \qquad x \in \Rd \setminus \left\{0\right\}. 
\end{align}
\end{lemma}
\begin{proof}
As before, we denote $\lambda:=|E|=-E >0$. Let $\eps \in (0,\lambda \land 1)$. We first prove that there are $R \geq 1$ and $M \in (0,1)$ such that
\begin{align}\label{Eq:EFest}
	|\varphi(x)| \le \sum_{k=1}^n \eps^{k-1} \int_{|y| \le R} g_{\lambda}^{\ast k}(x-y) V(y) |\varphi(y)| \d{y} + M^n \|\varphi_1\|_\infty, \quad  x \in \Rd \setminus \left\{0\right\}, \ n \in \NN, 
\end{align}
where $\varphi_1(y):= \I_{\left\{|y|>1\right\}} \varphi(y).$ 
Clearly, $\|\varphi_1\|_\infty < \infty$ because of \eqref{eq:eigenfunction_representation} and the estimate in Lemma \ref{eq:smooth}.

We will use induction. By \eqref{Eq1:eigf_est} for every $R \geq 1$, we have
\begin{align}
|\varphi(x) &\le \int_{|y| \le R} g_\lambda(x-y) V(y) |\varphi(y)| \d{y} + \kappa R^{-\alpha} \int_{|y|>R} g_{\lambda}(x-y)|\varphi(y)| \d{y} \label{eq1:EFest}\\
&\le \int_{|y| \le R} g_\lambda(x-y) V(y) |\varphi(y)| \d{y} + \kappa R^{-\alpha} G_{\lambda}|\varphi_1|(x)\,. \notag
\end{align}
Put $R = 1 \vee \left(\frac{\kappa}{\eps}\right)^{1/\alpha}$ and $M = \frac{\eps}{\lambda} < 1$.
Then,
using the estimate $G_{\lambda}|\varphi_1|(x) \leq \|\varphi_1\|_\infty/\lambda$ and the inequality $\frac{\kappa}{R^\alpha} \leq \eps$, 
\begin{align*}
|\varphi(x) \le \int_{|y| \le R} g_\lambda(x-y) V(y) |\varphi(y)| \d{y} + M \|\varphi_1\|_\infty, \quad  x \in \Rd \setminus \left\{0\right\}.
\end{align*}
So \eqref{Eq:EFest} holds for $n=1$. Now, suppose that \eqref{Eq:EFest} holds for some $n\ge1$. Then, we can use \eqref{Eq:EFest} to estimate $|\varphi|$ under the second integral in \eqref{eq1:EFest}. By proceeding in that way, we get
\begin{align*}
|\varphi(x) &\le \int_{|y| \le R} g_\lambda(x-y) V(y) |\varphi(y)| \d{y} \\
& \ \ + \eps\int_{|y|>R} g_{\lambda}(x-y) \left(\sum_{k=1}^n \eps^{k-1} \int_{|z| \le R} g_{\lambda}^{\ast k}(y-z) V(z) |\varphi(z)| \d{z} + M^n \|\varphi\|_\infty\right) \d{y}\\
&\le \int_{|y| \le R} g_\lambda(x-y) V(y) |\varphi(y)| \d{y} \\
& \ \ + \sum_{k=1}^n \eps^{k} \int_{|z| \le R} g_{\lambda}^{\ast (k+1)}(x-z) V(z) |\varphi(z)| \d{z} + M^{n+1} \|\varphi\|_\infty\\
&= \sum_{k=1}^{n+1} \eps^{k-1} \int_{|z| \le R} g_{\lambda}^{\ast k}(x-z) V(z) |\varphi(z)| \d{z} + M^{n+1} \|\varphi\|_\infty.
\end{align*}
Hence, by induction, \eqref{Eq:EFest} holds for all $n \ge 1$ and $x \in \Rd \setminus \left\{0\right\}$.

Lemma \ref{lem:Glconv} and Tonelli's theorem lead us to a concluding estimate
\begin{align*}
|\varphi(x)| &\le \sum_{k=1}^n \eps^{k-1} \int_{|y| \le R} \int_0^\infty \frac{t^{k-1}}{(k-1)!}e^{-\lambda t} p_t(x-y)  V(y) |\varphi(y)| \d{t}\d{y} + M^n \|\varphi\|_\infty \\
&\le \int_{|y| \le R} \int_0^\infty \sum_{k=1}^\infty \frac{(\eps t)^{k-1}}{(k-1)!}e^{-\lambda t} p_t(x-y)  V(y) |\varphi(y)| \d{t}\d{y} + M^n \|\varphi\|_\infty \\
&= \int_{|y| \le R}  g_{\lambda-\eps}(x-y)  V(y) |\varphi(y)| \d{y} + M^n \|\varphi\|_\infty.
\end{align*}
By letting $n \to \infty$, we obtain the claimed bound \eqref{Eqcor:EFest}. 
\end{proof}

The following theorem summarizes our investigations in this section. It leads to the estimates for eigenfunctions at infinity which are given in terms of the resolvent kernels of the operator $L$. 
 
\begin{theorem}\label{th:EFest}
Let \eqref{A1} hold and let $\kappa< \kappa^*$. If $\varphi \in L^2(\R^d)$ is such that \eqref{eq:eigenfunction} holds,
then for every $\eps \in (0,|E| \land 1)$ there is $R=R(\eps) \geq 1$ and $c=c(\eps)$ such that
\begin{align}\label{eqth:EFest}
	|\varphi(x)| \le c \sup_{|y| \le R} g_{|E|-\eps}(x-y), \qquad x \in \Rd \setminus \left\{0\right\}. 
\end{align}
Furthermore, if $\varphi$ is a ground state (i.e.\ $E = \inf \spec(H) <0$), then there is $\widetilde c>0$ such that
\begin{align}\label{eqth:EFest}
	\varphi(x) \ge \widetilde c \inf_{|y| \le 1} g_{|E|}(x-y), \qquad x \in \Rd \setminus \left\{0\right\}. 
\end{align}
\end{theorem}

\begin{proof}
The upper bound holds by Lemma \ref{lem:EFest} and the lower estimate follows from the second assertion of Lemma \ref{lem:eigf_est}.
\end{proof} 

\section{Estimates of heat kernels}

Throughout this section we assume that $\kappa \leq \kappa^*$, i.e.\ we allow for critical potential $V(x) = \kappa^*|x|^{-\alpha}$. Recall that $\nu(dx) = \nu(x)dx$ is a symmetric L\'evy measure such that the assumption \eqref{A1} holds. The heat kernel of the operator $L$ is given by $p(t,x,y):= p_t(y-x)$, where $p_t$ is a probability density function given by \eqref{eq:free_density}. Hence, for every $t>0$, $p_t$ is a symmetric, bounded and continuous function such that 
\[
\int_{\R^d} e^{i \xi \cdot y} p_t(y) \d y= e^{-t \psi(\xi)}, \quad t>0, \ \xi \in \R^d,
\]
where
\[
\psi(\xi) = \int_{\R^d \setminus\{0\}} (1-\cos(\xi \cdot y) \nu(dy), \quad \xi \in \R^d.
\]

\subsection{Upper estimates of the perturbed heat kernel}
In this section we prove the upper estimates for the kernel $\tp(t,x,y)$ which was constructed in Section \ref{sec:pert_kernel}. This will be done for L\'evy measures with densities that satisfy our both assumptions \eqref{A1} and \eqref{A2}. 

Let
\[
\psi^*(u):= \sup_{|\xi| \leq u} \psi(\xi), \quad u \geq 0,
\] 
be a maximal function of $\psi$.
Recall that $\psi(\xi) \asymp |\xi|^{\alpha}$, for $|\xi| \geq (2|\sigma|)^{1/\alpha}$, see \eqref{eq:LK_comp}. Since
\begin{align} \label{eq:maximal}
\psi(\xi) \asymp \psi^*(|\xi|), \quad \xi \in \R^d,
\end{align}
(see \cite[Lemma 5(a)]{KK-PS-2017}), it extends to
\begin{align} \label{eq:psi_stable}
\psi(\xi) \asymp |\xi|^{\alpha}, \quad |\xi| \geq r,
\end{align} 
for every $r>0$, with comparability constant depending on $r$. An important consequence of \eqref{eq:psi_stable} is that for every $r>0$ there is a constant $c>0$ such that
\begin{align} \label{eq:nu_by_psi}
\nu(x) \geq c \frac{\psi^*(1/|x|)}{|x|^d}, \quad |x| \leq r.
\end{align} 
This was originally established in \cite[Theorem 26]{KB-TG-MR-2014} for radial decreasing densities $\nu(x)$, but due to \eqref{eq:profile} it easily extends to our setting, see \cite[Lemma 5(b)]{KK-PS-2017} or \cite[Lemat 2.3]{TG-KS-1}, \cite[Lemat A.3]{TG-KS-2}. 

We first collect the properties of $\nu(x)$ and $p(t,x,y)$ that are needed below.

	\begin{lemma} \label{lem:properties_pt_nu} 
		Under assumptions \textup{\eqref{A1}} and \textup{\eqref{A2}}, we have the following statements.
		\begin{itemize}
		  \item[(a)]  For every $r>0$ there exists a constant $c>0$ such that 
		  $$
			c \nua(x) \leq \nur(x) \leq \nua(x), \quad |x| \leq r.
			$$
			\item[(b)] For every $r>0$ there exists a constant $c>0$ such that
			\begin{align}\label{eq:loc_comp}
			\nur(y) \leq c \nu(x), \quad |y| \geq 1, \ |y-x| \leq r. 
			\end{align}
			\item[(c)]  There exists a constant $c>0$ such that
			$$
			\nur(y) \nur(x) \leq c \nur(y-x), \quad |x|, |y| \geq 1.
			$$
			\item[(d)] For every $T, K>0$ there exist the constants $c, \widetilde c$ (depending on $T$ and $K$) such that
			\[
			p(t,x,y) \stackrel{c}{\asymp} t^{-d/\alpha}, \quad t \in (0,T], \ |y-x| \leq Kt^{1/\alpha},
			\]
			and
			\[
			p(t,x,y) \stackrel{\widetilde c}{\asymp} t \nu(y-x), \quad t \in (0,T], \  |y-x| \geq Kt^{1/\alpha}.
			\]
			In particular,
	$$
			p(t,x,y) \asymp t^{-d/\alpha} \wedge t \nu(y-x), \quad t \in (0,T], \ x,y \in \R^d.
			$$
			\item[(e)] For every $T, K>0$ there exists a constant $c>0$ (depending on $T$ and $K$) such that
			\[
			p(t,z,y) \geq c p(t,x,y), \quad x,y,z \in \R^d,\ |x-z| \leq K t^{1/\alpha}, \ t \in (0,T]. 
			\]
		\end{itemize}
	\end{lemma}
	\begin{proof} (a) The upper bound is a direct consequence of \eqref{A1}. The lower bound follows directly from \eqref{eq:nu_by_psi}, \eqref{eq:maximal} and \eqref{eq:psi_stable}. 
	
\smallskip	
\noindent
 (b) Due to \eqref{eq:profile} the assertion is trivial for $|x| < 1$. We are left to consider $|x| \geq 1$. It follows from a combination of \cite[Lemmas 1 and 3]{KK-PS-2017}, but we provide here a short and direct proof for reader's convenience. First recall that by Assumption \eqref{A2} there exists a decreasing profile $f$ of the density $\nu$ such that
\begin{align}\label{eq:DJP_aux}
\int_{\R^d} f_1(|x-y|)f_1(|y|) \d y \leq c_1 f_1(|x|), \quad x \in \R^d, 
\end{align}
for a constant $c_1$. Here $f_1 = f  \wedge 1$. Moreover, observe that
\[
c_2 f(|x|) \leq f_1(|x|) \leq f(|x|), \quad |x| \geq 1,
\]
where $c_2 = 1/(1 \vee f(1))$. Therefore, in order to complete the proof of \eqref{eq:loc_comp}, we only need to show that for any $r>0$ there exists a constant $c_3=c_3(r)>0$ such that 
\[
f_1(u) \leq c_3 f_1(u+r), \quad u \geq 1.
\]
Fix $r>0$. Let $x = (u+r,0,\ldots,0)$, $u \geq 1$. By \eqref{eq:DJP_aux} we have
\[ 
c_1 f_1(u+r) = c_1 f_1(|x|) \geq \int_{|y-x|<r} f_1(|x-y|)f_1(|y|) \d y \geq  f_1(|x|-r) \int_{|y|<r} f_1(|y|)\d y.
\]
Since $|x|-r = u$ and $\int_{|y|<r} f_1(|y|)\d y>0$, this completes the proof of part (b).

\smallskip	
\noindent	
(c) This easily follows from the last assertion of \cite[Lemma 3.1]{KK-RS-2020}.

\smallskip	
\noindent	
(d) By the first assertion of \cite[Lemma 3.1]{KK-RS-2020}, \eqref{eq:DJP_aux} is equivalent to the condition that for every $r>0$ there is a constant $c_4 = c_4(r)>0$ such that
\[
\int_{|y-x|>r \atop |y|>r} \nu(x-y)\nu(y) \d y \leq c_4 \nu(x), \quad |x| \geq r
\]
(this is stated for $r=1$, but the proof for arbitrary $r$ is the same -- this is a consequence of radiality and monotonicity of the profile $f$). This inequality and \eqref{eq:nu_by_psi} shows that the assumptions (1.1) of \cite[Theorem 1]{KK-PS-2017} are satisified for every $r_0>0$ (according to the notation in the quoted paper).
This theorem gives that for every $T>0$ there exist constants $c_5, c_6$ (depending on $T$) and $\theta>0$ such that
\[
p(t,x,y) \stackrel{c_5}{\asymp} t^{-d/\alpha}, \quad t \in (0,T], \ |y-x| \leq \theta t^{1/\alpha},
\]
and
\[
p(t,x,y) \stackrel{c_6}{\asymp} t \nu(y-x), \quad t \in (0,T], \ |y-x| \geq \theta t^{1/\alpha}.
\]
By part (a), for any fixed $r>0$ (in particular, $r = K T^{1/\alpha}$ for any $K>0$) we have $\nu(x) \asymp |x|^{-d-\alpha}$, $|x| \leq r$ . 
Hence, we see that 
	\[
			p(t,x,y) \asymp t^{-d/\alpha} \wedge t \nu(y-x), \quad t \in (0,T], \ x, y \in \R^d,
	\]
and the similar comparabilities hold with $\theta = K$, for any $K>0$.

	We remark that the original statement of \cite[Theorem 1]{KK-PS-2017} says that we have the estimates for \emph{some $T>0$}. However, in our case the assumptions are satisfied for every $r_0>0$ and therefore the estimates hold \emph{for every $T>0$}. This can be directly checked by inspecting the proof of that theorem.  
	
\smallskip	
\noindent	
(e) This bound can be obtained e.g.\ by iterating the estimate in \cite[Corollary 4.2]{TG-KK-PS-2022} (with $t_0=T$).
	\end{proof}

\noindent For $t>0$ and $x \in \Rdo$ we denote
\begin{align}
	H(t,x) = 1+ (t^{-1/\alpha}|x|)^{-\delta}.
\end{align}
It follows from \cite[Proposition 3.2]{KB-TG-TJ-DP-2019} and \eqref{eq:est_by_stable_hardy} that there exists a constant $c>0$ such that
\begin{align} \label{eq:upper_est_by_H}
	\int_{\R^d}\tp(t,x,y) \d{y} \leq ce^{|\sigma|t} H(t,x), \quad t>0, \ \ x \in \R^d \setminus \left\{0\right\}.
\end{align}
\begin{lemma} \label{lem:small_args}
For every $T>0$ and $R>0$ there is a constant $c > 0$ such that 
\begin{align}
	\tp(t,x,y) \le c\, H(t,x) H(t,y) p(t,x,y) , \qquad |x|, |y| \leq R, \ t \in (0,T].
\end{align}
\end{lemma}
\begin{proof}
Fix $T, R >0$. By \eqref{eq:est_by_stable_hardy} and \eqref{eq:tilde_pa_est}, for $t \in (0,T]$ and $x, y \in \R^d$, we have
\begin{align*}
\tp(t,x,y) \le c_1 \tpa(t,x,y) &\le c_2 H(t,x) H(t,y) \left(t^{-d/\alpha} \land \frac{t}{|x-y|^{d+\alpha}}\right) .
\end{align*}
Since $|x-y|\leq |x|+|y|\leq 2R$, Lemma \ref{lem:properties_pt_nu} (a) implies that there is $c_3=c_3(R)$ such that
\begin{align*}
t^{-d/\alpha} \land \frac{t}{|x-y|^{d+\alpha}} \leq c_3 \left(t^{-d/\alpha} \land t \nu(x-y) \right).
\end{align*}
Together with Lemma \ref{lem:properties_pt_nu} (d), this gives the claimed bound.
\end{proof}

\medskip

Observe now that for any $t>0$, $x,y \in\R^d$ and $R>0$, by \eqref{eq:perturbation}, we have
\begin{align*}
\tp(t,x,y) &\le p(t,x,y) + \int_0^t \int_{|z|<R} \tp(t-s,x,z) V(z) p(s,z,y) \d{z}\d{s} \\
& \ \ \ \ + \int_0^t \int_{|z|>R} \tp(t-s,x,z) \kappa R^{-\alpha} p(s,z,y) \d{z}\d{s} \\
&\le p(t,x,y) + \int_0^t \int_{|z|<R} \tp(t-s,x,z) V(z) p(s,z,y) \d{z}\d{s}  + t\kappa R^{-\alpha} \tp(t,x,y).\\
\end{align*}
Hence, if $t\kappa R^{-\alpha} < 1$, then
\begin{align*}
\tp(t,x,y) &\le \frac{1}{1- t\kappa R^{-\alpha}} \left( p(t,x,y) + \int_0^t \int_{|z|<R} \tp(t-s,x,z) V(z) p(s,z,y) \d{z}\d{s} \right).
\end{align*}
For any fixed $T>0$ we set
\begin{align} \label{eq:R0_def}
R_0 = R_0(T) := 1 \vee (2T\kappa)^{1/\alpha}.
\end{align}
Then, 
\[
\frac{1}{1- t\kappa R_0^{-\alpha}}<2, \quad t \in (0,T],
\]
and, consequently, 
\begin{align}\label{eq:Duhamelcut}
\tp(t,x,y) &\le 2 \left( p(t,x,y) + \int_0^t \int_{|z|<R_0} \tp(t-s,x,z) V(z) p(s,z,y) \d{z}\d{s} \right).
\end{align}

\begin{lemma}\label{lem:onelarge}
Let $T>0$ and let $R_0=R_0(T)$ be a number given by \eqref{eq:R0_def}. Then there is a constant $c=c(T)$ such that for any $|x|\le R_0+1$, $|y|\ge R_0+2$ and $t \in (0,T]$ we have 
\begin{align}\label{Eq:onelarge}
	\tp(t,x,y) \le cH(t,x) t \nur(y-x) .
\end{align}
\end{lemma}
\begin{proof}
First observe that if $|z|<R_0$, then $|y-z| >1$ and $|(y-x) - (y-z)| = |x-z| < 2R_0+1$. Hence, by using Lemma \ref{lem:properties_pt_nu} (b), we get $\nur(y-z) \le c_1 \nu(y-x)$, for a positive constant $c_1$. Together with the estimate in Lemma \ref{lem:properties_pt_nu} (d), this implies that 
\begin{align}\label{eq:ps_aux}
p(s,z,y) \le c_2 s \nu(y-x), \quad |z| < R_0, \ s \in (0,T],
\end{align}
for some $c_2=c_2(T)$.
By \eqref{eq:Duhamelcut} and \eqref{eq:ps_aux}, we get
\begin{align*}
\tp(t,x,y) &\le 2p(t,x,y) + 2c_2 t\nur(y-x)\int_0^t \int_{|z|<R_0} \tp(t-s,x,z)V(z)\d{z} \d{s}\\
           &\le 2p(t,x,y) + 2c_2 t\nur(y-x)\int_0^t \int_{\R^d} \tp(t-s,x,z)V(z)\d{z} \d{s}.
\end{align*}
On the other hand, by integrating on both sides of the equality \eqref{eq:perturbation}, we obtain
\[
\int_{\R^d} \tp(t,x,y) \d{y} = 1 + \int_0^t \int_{\R^d} \tp(t-s,x,z) V(z) \d{z} \d{s},
\]
and further, by using \eqref{eq:upper_est_by_H}, we get
\[
\int_0^t \int_{\R^d} \tp(t-s,x,z) V(z) \d{z} \d{s} \leq c_3 H(t,x),
\]
for a constant $c_3=c_3(T)$. This estimate, Lemma \ref{lem:properties_pt_nu} (d) and the inequality $H(t,x) \geq 1$ lead us to a conclusion 
\begin{align*}
\tp(t,x,y) &\le 2p(t,x,y) + 2c_2c_3t\nur(y-x) H(t,x) \leq c_4 H(t,x) t \nur(y-x).
\end{align*}
\end{proof}

We are now in a position to make a concluding step in this section.

\begin{lemma} \label{lem:final_step}
Let $T>0$ and let $R_0=R_0(T)$ be a number given by \eqref{eq:R0_def}. Then there is a constant $c=c(T)$ such that for any $x, y \in \R^d$ satisfying the condition $|x| \vee |y| \ge R_0+2$ and $t \in (0,T]$ we have
\begin{align}\label{Eq:xylarge}
	\tp(t,x,y) \le c H(t,x)H(t,y) t \nur(y-x) .
\end{align} 
\end{lemma}
\begin{proof}
By symmetry of the kernel $\tp(t,x,y)$ and Lemma \ref{lem:onelarge} it suffices to consider only the case $R_0+1 < |x| \le |y|$. By \eqref{eq:Duhamelcut} and Lemmas \ref{lem:properties_pt_nu} (d) and \ref{lem:onelarge}, we have
\begin{align*}
\tp(t,x,y) &\le 2p(t,x,y) + 2\int_0^t \int_{|z|<R_0} p(t-s,x,z)V(z)\tp(s,z,y) \d{z} \d{s}\\
&\le c_1t\nur(y-x) + c_1 \int_0^t \int_{|z|<R_0} (t-s)\nur(z-x)  V(z) s\nur(y-z) H(s,z) \d{z} \d{s}\\
&\le c_1t\nur(y-x) + c_1 t^2\int_0^t \int_{|z|<R_0} \nur(y-z) \nur(z-x) V(z)H(s,z) \d{z} \d{s}.
\end{align*}
Since for $|z|<R_0$ we have $|y-z|, |z-x| > 1$, Lemma \ref{lem:properties_pt_nu} (c) implies that
\begin{align*}
\nur(y-z) \nur(z-x) \le c_2 \nur(y-x). 
\end{align*}
Hence,
\begin{align*}
\tp(t,x,y) &\le \left(c_1+c_1 T \int_0^t \int_{|z|<R_0} V(z)H(s,z) \d{z} \d{s} \right) t\nur(y-x).
\end{align*}
By definition of $H$, 
\[
\int_0^t \int_{|z|<R_0} V(z)H(s,z) \d{z} \d{s} \leq \kappa \left(T  \int_{|z|<R_0} |z|^{-\alpha} \d{z} +T^{1+\delta/\alpha} \int_{|z|<R_0} |z|^{-\alpha-\delta} \d{z}\right).
\]
Recall that $\alpha<d$, which implies that the first integral on the right hand side is finite. Since $\alpha \leq \alpha+\delta \leq \alpha + (d-\alpha)/2 = (d+\alpha)/2 < d$, the second integral converges as well.
Hence,
\begin{align*}
\tp(t,x,y) \le c_3 t \nur(y-x) \le c_3 H(t,x) H(t,y) t \nur(y-x)
\end{align*}
with $c_3=c_3(T)$, which completes the proof.
\end{proof}

\begin{proof}[Proof of Theorem \ref{th:heat_kernel}]
Fix $T>0$ and $R_0=R_0(T)$ as in \eqref{eq:R0_def}. If $|x|,|y| \leq R_0+2$, then the claimed estimate follows from Lemma \ref{lem:small_args} with $R=R_0+2$. If $|x| \vee |y| \ge R_0+2$, then by \eqref{eq:est_by_stable_hardy} and \eqref{eq:tilde_pa_est} we have
\[
\tp(t,x,y) \leq e^{|\sigma|T} \tpa(t,x,y) \leq c_1 H(t,x)H(t,y) \, t^{-d/\alpha}
\]
and, by Lemma \ref{lem:final_step},
\[
	\tp(t,x,y) \le c_2 H(t,x)H(t,y) \, t \, \nur(y-x),
\]
for every $t \in (0,T]$. The result follows then from Lemma \ref{lem:properties_pt_nu} (d). 
\end{proof}

\subsection{Lower estimate of the perturbed heat kernel and the ground state} \label{sec:last_result}

We will now prove our last main theorem. 

\begin{proof}[Proof of Theorem \ref{th:heat_kernel_and_ground_state}] 
(a) $\Rightarrow$ (b) Fix $R>0$ and take $T=(2R)^{\alpha}$. By \eqref{eq:eigenfunction_representation} and the estimate in part (a) applied to $|x| \leq T^{1/\alpha}$ and Lemma \ref{lem:properties_pt_nu} (d), 
\begin{align*}
\varphi(x) & = e^{ET} \int_{\R^d} \tp(T,x,y)\varphi(y) \d{y} \\
           & \geq c \frac{e^{ET}T^{\delta/\alpha}}{|x|^{\delta}} \int_{|y| < \frac{T^{1/\alpha}}{2}} p(T,x,y) \varphi(y) \d{y} \\
					 & \geq c_1 \frac{e^{ET}T^{(\delta-d)/\alpha}}{|x|^{\delta}} \int_{|y| < R} \varphi(y) \d{y} \\
					 & = c_1 \frac{e^{E(2R)^{\alpha}}}{(2R)^{d-\delta}} \left( \int_{|y| < R} \varphi(y) \d{y}\right) \frac{1}{|x|^{\delta}}
					   =  \frac{c_2}{|x|^{\delta}}, \quad |x|  \leq R.
\end{align*}
This gives the estimate in (b).

\smallskip
\noindent
(b) $\Rightarrow$ (a) We adapt the argument from \cite[Section 4.2]{KB-TG-TJ-DP-2019}. First we fix the notation: 
\[
\phi_t(x) = 1+ t^{\delta/\alpha}\varphi(x), \qquad \mu_t(\d z)=\phi_t^2(z)\d z, \qquad q(t,x,y)=\frac{\tp(t,x,y)}{\phi_t(x) \phi_t(y)}.
\]
Fix $T>0$. By \eqref{eq:eigenfunction_representation}, we have for $t \in (0,T]$, $x \in \R^d \setminus \left\{0\right\}$,
\begin{align} \label{eq:q_frombelow}
\int_{\R^d} q(t,x,z) \mu_t(\d z) = \frac{\int_{\R^d} \tp(t,x,z) \d z + t^{\delta/\alpha} e^{-Et} \varphi(x)}{\phi_t(x)} 
                                \geq \frac{1 + t^{\delta/\alpha} \varphi(x)}{\phi_t(x)} = 1. 
\end{align}
Recall that $H(t,x) = 1+ t^{\delta/\alpha}|x|^{-\delta}$. Observe that for $|x| \geq T^{1/\alpha}$ we have
\begin{align} \label{eq:gs_H_1}
\frac{H(t,x)}{\phi_t(x)} \leq 1+T^{\delta/\alpha}|x|^{-\delta} \leq 2,
\end{align}
while for $|x| \leq T^{1/\alpha}$,
\begin{align} \label{eq:gs_H_2}
\frac{H(t,x)}{\phi_t(x)} = \frac{1+ t^{\delta/\alpha}|x|^{-\delta}}{1+ t^{\delta/\alpha}\varphi(x)}\leq c_3,
\end{align}
by the estimate of the ground state in (b). Together with Theorem \ref{th:heat_kernel} and \eqref{eq:domination_free_kernels} this implies that 
\[
q(t,x,z) \leq c_4 p(t,x,z) \leq c_4 e^{|\sigma|T} p^{(\alpha)}(t,x,z), \quad t \in (0,T], \ x,z \in \R^d \setminus \left\{0\right\}.
\]
Hence, by using the scaling property \eqref{eq:sca}, \eqref{eq:q_frombelow} and by following the argument in \cite[Section 4.2, p.\ 38]{KB-TG-TJ-DP-2019}, we obtain that there are $r \in (0,1)$ and $R >1+4^{1/\alpha}$ such that
\begin{align} \label{eq:last_aux_1}
\int_{r \leq \frac{|z|}{t^{1/\alpha}} \leq R } q(t,x,z) \mu_t(\d z) \geq \frac{1}{2}, \quad |x| \leq (4t)^{1/\alpha}, \ x \neq 0, \ t \in (0,T]. 
\end{align}
We are now ready to give the proof of (a). By the symmetry of the kernel $\tp(t,x,y)$, it is enough to establish the estimate for $|x| \leq t^{1/\alpha}$, $t \in (0,T]$, $y \in \R^d$. 

Assume first that $|x| \leq (4t)^{1/\alpha}$ and $|y| \geq r(2t)^{1/\alpha}$. By the Chapman-Kolmogorov equation \eqref{eq:Ch-K},
\[
q(2t,x,y) \geq c_5 \int_{r \leq \frac{|z|}{t^{1/\alpha}} \leq R } q(t,x,z)q(t,z,y) \mu_t(\d z) \geq c_5\int_{r \leq \frac{|z|}{t^{1/\alpha}} \leq R } q(t,x,z)\frac{p(t,z,y)}{\phi_t(z)\phi_t(y)} \mu_t(\d z).
\] 
By Lemma \ref{lem:properties_pt_nu} (e) (here we use that $|x-z| \leq (4^{1/\alpha} + R) t^{1/\alpha}$) 
and Proposition \ref{lem:est_at_zero}, 
\[
\frac{p(t,z,y)}{\phi_t(z)\phi_t(y)} 
\geq c_6 \frac{p(t,x,y)}{(1+c_7 t^{\delta/\alpha} (1+|z|^{-\delta}))(1+c_7 t^{\delta/\alpha} (1+|y|^{-\delta}))} 
\geq c_6 \frac{p(t,x,y)}{(1+c_7 T^{\delta/\alpha} +r^{-\delta})^2}.
\]
Together with \eqref{eq:last_aux_1}, this implies that
\[
q(2t,x,y) \geq \frac{c_8 }{2(1+c_7 T^{\delta/\alpha} +r^{-\delta})^2}p(t,x,y),
\]
and, by using the two-sided sharp estimates in Lemma \ref{lem:properties_pt_nu} (d) 
(giving $p(t,x,y) \asymp p(2t,x,y)$), we see that
\begin{align} \label{eq:last_aux_2}
q(2t,x,y) \geq c_9 p(2t,x,y), \quad |x| \leq (4t)^{1/\alpha}, \ |y| \geq r(2t)^{1/\alpha}, \ t \in (0,T].
\end{align}
In particular,
\begin{align} \label{eq:last_aux_3}
q(t,x,y) \geq c_9 p(t,x,y), \quad |x| \leq (2t)^{1/\alpha}, \ |y| \geq rt^{1/\alpha}, \ t \in (0,T].
\end{align} 
Consider now the case $|x| \leq (2t)^{1/\alpha}$ and $|y| \leq r(2t)^{1/\alpha}$ (in particular, $|y| < (2t)^{1/\alpha}$). By the symmetry of the kernel $q(t,x,y)$ and \eqref{eq:last_aux_3}, we get
\[
q(2t,x,y) \geq c_5 \int_{r \leq \frac{|z|}{t^{1/\alpha}} \leq R } q(t,x,z)q(t,z,y) \mu_t(\d z) \geq c_{10} \int_{r \leq \frac{|z|}{t^{1/\alpha}} \leq R } q(t,x,z)p(t,z,y)\mu_t(\d z).
\]
One more use of Lemma \ref{lem:properties_pt_nu} (e), (d) and \eqref{eq:last_aux_1} as above leads us to the estimate
\[
q(2t,x,y) \geq c_{11} p(2t,x,y), \quad |x| \leq (2t)^{1/\alpha},\ |y| \leq r(2t)^{1/\alpha}, \ t \in (0,T].
\]
Finally, by combining this with \eqref{eq:last_aux_2} and by using \eqref{eq:gs_H_1}--\eqref{eq:gs_H_2}, we obtain the estimate in (a) and complete the proof of the theorem.
\end{proof}

\section{Applications to relativistic Coulomb model} \label{sec:relativistic}

Throughout this section we assume that $d= 3$, $\alpha = 1$ and $m>0$.  
We are now in a position to apply our results to relativistic Coulomb model. 
Let
\begin{align} \label{eq:relativistic_Coulomb}
H_m = \sqrt{-\Delta+m^2} - \frac{\kappa}{|x|}, \quad \text{where} \quad 0 < \kappa < \kappa^{*} := \frac{2}{\pi},
\end{align}
and let $\delta$ be the unique number such that 
\[
0 < \delta <1 \quad \text{and} \quad \kappa = \frac{2 \Gamma\left(\frac{1+\delta}{2}\right) \Gamma\left(\frac{3-\delta}{2}\right)}{\Gamma\left(\frac{\delta}{2}\right)\Gamma\left(1-\frac{\delta}{2}\right)} = (1-\delta) \tan\frac{\pi \delta}{2}. 
\]
Herbst \cite{IH-1977} and Weder \cite{Weder-1975} studied the structure of the spectrum of the operator $H_m$. It is known that $\spec_e(H_m) = [m,\infty)$,
\[
\spec_d(H_m) \subset \left[m \sqrt{1-\left(\frac{\kappa \pi}{2}\right)^2}, m\right),
\]
and $\spec_d(H_m)$ is infinite. 
The ground state eigenvalue $E_{m,0} := \inf \spec(H_m) = \inf \spec_d(H_m)$ is simple and the corresponding ground state eigenfunction $\varphi_{m,0}$ is strictly positive, see the discussion in Daubechies and Lieb \cite{ID-EL-1983}, and Daubechies \cite{ID-1984}.

\begin{corollary}\label{cor:relativistic}
Let $H_m$ be given by \eqref{eq:relativistic_Coulomb}. Any $L^2$-eigenfunction of $H_m$ is continuous on $\R^3 \setminus \left\{0\right\}$ and the following pointwise estimates hold.

\begin{itemize}
\item[(a)] \textup{(General upper estimate)} For any eigenfunction $\varphi_m$ of $H_m$ corresponding to eigenvalue $E_m < m$ there exists a constant $c >0$ such that
\[
|\varphi_m(x)| \le \frac{c }{|x|^{\delta}}, \quad 0 < |x| \leq 1,
\]
and for every $\eps >0$ there exists a constant $\widetilde c=\widetilde c(\eps) >0$ such that
\[
	|\varphi_m(x)| \le \widetilde c \, e^{-\big(\sqrt{m^2 - E_m^2} - \eps\big)|x|}, \qquad |x| \geq 1.
\]
\item[(b)] \textup{(Two-sided estimate of the ground state)} There exists $c \geq 1$ such that
\[
\frac{1}{c} \frac{1}{|x|^{\delta}} \leq \varphi_{m,0}(x) \leq \frac{c}{|x|^{\delta}}, \quad 0<|x| \leq 1, 
\]
and for every $\eps >0$ there exists $\widetilde c=\widetilde c(\eps) \ge 1$ such that
\[
\frac{1}{\widetilde c}e^{-\big(\sqrt{m^2 - E_{m,0}^2} + \eps\big)|x|}	\le \varphi_{m,0}(x) \le \widetilde c \, e^{-\big(\sqrt{m^2 - E_{m,0}^2} -\eps\big)|x|}, \quad |x| \geq 1. 
\]
\end{itemize}
\end{corollary} 

\begin{proof}
We apply our results to the operator $H=H_m-m$. In particular, if $\varphi_m$ is an eigenfunction of $H_m$ corresponding to an eigenvalue $E_m < m$, then $H \varphi_m = E \varphi_m$, where $E = E_m-m <0$.

We first show the upper bound for $|x| \geq 1$. Fix $\eps > 0$. By the estimates of the resolvent kernel $g_{\lambda}$ established in \cite[(II.19)--(II.20) on p.\ 128--129]{RC-WM-BS-1990} we obtain that for every $\lambda \in (0,m)$ there is a constant $c_1=c_1(\eps, \lambda) \geq 1$ such that
\begin{align}\label{eq:resolvent_est}
c_1^{-1} e^{-(\sqrt{2m\lambda - \lambda^2} + \eps/2)  |x|} \leq g_{\lambda}(x) \leq c_1 e^{-(\sqrt{2m\lambda - \lambda^2} - \eps/2)  |x|},  \quad |x| \geq 1.
\end{align}
Moreover, we observe that $\sqrt{2m\lambda - \lambda^2} = \sqrt{m^2 - (m-\lambda)^2}$, $|E| = m-E_m$, and that by the continuity of the map $[0,m] \ni u \mapsto \sqrt{m^2 - u^2}$ we have the following:\ there is $\bar{\eps} \in \big(0,(m-E_m)\wedge 1 \big)$ such that
\[
\sqrt{m^2 - (m-|E|+\bar{\eps})^2} = \sqrt{m^2 - (E_m+\bar{\eps})^2} \geq \sqrt{m^2 - E_m^2} - \eps/2.
\]
Hence, by the upper bounds in Theorem \ref{th:bound_states_main} (b) and \eqref{eq:resolvent_est}, there exists $R = R(\bar{\eps})$ such that for every $|x| \geq R+1$ we have
\[
|\varphi_m(x)| \leq c_2 \sup_{|y| \leq R} g_{|E|-\bar{\eps}}(x-y) \leq c_3 e^{-(\sqrt{m^2 - (m-|E|+\bar{\eps})^2} - \eps/2)(|x|-R)} 
               \leq c_4 e^{-(\sqrt{m^2 - E_m^2} - \eps)|x|}.
\]
On the other hand, since $\varphi_m$ is continuous on $\R^d \setminus \left\{0\right\}$, for $1 \leq |x| \leq R+1$ we get
\[
|\varphi_m(x)| \leq c_5 \sup_{1 \leq |x| \leq R+1} |\varphi_m(x)| \leq c_6 e^{-(\sqrt{m^2 - E_m^2} - \eps)|x|},
\] 
with $c_6 := c_5 \sup_{1 \leq |x| \leq R+1} |\varphi_m(x)| e^{\sqrt{m^2 - E_m^2}(R+1)}$. 

The lower bound of the ground state for $|x| \geq 2$ follows from the lower estimates in 
Theorem \ref{th:bound_states_main} (b) and \eqref{eq:resolvent_est} by similar arguments. For $1 \leq |x| \leq 2$ we just use the continuity and the strict positivity of the ground state, and the inequality $\exp(-\sqrt{m^2 - E_m^2}) \leq 1$. 

The estimates for $|x| \leq 1$ follow directly from Theorem \ref{th:bound_states_main} (a), Theorem \ref{th:heat_kernel_and_ground_state} and the lower estimate of the heat kernel $\tp(t,x,y)$ in \cite[Theorem 5.1]{TJ-KK-KS-2022} (our $\tp(t,x,y)$ is $e^{mt} \tilde p^m_{V^0_{\delta}}(t,x,y)$ in the notation of the quoted paper). 
\end{proof}

\begin{remark}
In principle, the constants $c$ and $\widetilde c$ in the above corollary depend on parameters $m$ and $\kappa$, as well as on $E_m$ and $\varphi_m$. The statements and proofs of our main theorems allow one to track this dependence up to some reasonable level. 
\end{remark}


\begin{thebibliography}{10}

\bibitem{SA-1982}
S.~Agmon.
\newblock {\em Lectures on exponential decay of solutions of second-order
  elliptic equations: bounds on eigenfunctions of {$N$}-body {S}chr\"{o}dinger
  operators}, volume~29 of {\em Mathematical Notes}.
\newblock Princeton University Press, Princeton, NJ; University of Tokyo Press,
  Tokyo, 1982.

\bibitem{SA-1985}
S.~Agmon.
\newblock Bounds on exponential decay of eigenfunctions of {S}chr\"{o}dinger
  operators.
\newblock In {\em Schr\"{o}dinger operators ({C}omo, 1984)}, volume 1159 of
  {\em Lecture Notes in Math.}, pages 1--38. Springer, Berlin, 1985.

\bibitem{GA-JL-2022}
G.~Ascione and J.~L\H{o}rinczi.
\newblock Potentials for non-local {S}chr\"{o}dinger operators with zero
  eigenvalues.
\newblock {\em J. Differential Equations}, 317:264--364, 2022.

\bibitem{KB-BD-PK-2016}
K.~Bogdan, B.~Dyda, and P.~Kim.
\newblock Hardy inequalities and non-explosion results for semigroups.
\newblock {\em Potential Anal.}, 44(2):229--247, 2016.

\bibitem{KB-TG-TJ-DP-2019}
K.~Bogdan, T.~Grzywny, T.~Jakubowski, and D.~Pilarczyk.
\newblock Fractional {L}aplacian with {H}ardy potential.
\newblock {\em Comm. Partial Differential Equations}, 44(1):20--50, 2019.

\bibitem{KB-TG-MR-2014}
K.~Bogdan, T.~Grzywny, and M.~Ryznar.
\newblock Density and tails of unimodal convolution semigroups.
\newblock {\em J. Funct. Anal.}, 266(6):3543--3571, 2014.

\bibitem{KB-WH-TJ-2008}
K.~Bogdan, W.~Hansen, and T.~Jakubowski.
\newblock Time-dependent {S}chr\"{o}dinger perturbations of transition
  densities.
\newblock {\em Studia Math.}, 189(3):235--254, 2008.

\bibitem{Bottcher_Schilling_Wang}
B.~B\"{o}ttcher, R.~Schilling, and J.~Wang.
\newblock {\em L\'{e}vy matters. {III}}, volume 2099 of {\em Lecture Notes in
  Mathematics}.
\newblock Springer, Cham, 2013.
\newblock L\'{e}vy-type processes: construction, approximation and sample path
  properties, With a short biography of Paul L\'{e}vy by Jean Jacod, L\'{e}vy
  Matters.

\bibitem{RC-WM-BS-1990}
R.~Carmona, W.~C. Masters, and B.~Simon.
\newblock Relativistic {S}chr\"{o}dinger operators: asymptotic behavior of the
  eigenfunctions.
\newblock {\em J. Funct. Anal.}, 91(1):117--142, 1990.

\bibitem{ID-1984}
I.~Daubechies.
\newblock One-electron molecules with relativistic kinetic energy: properties
  of the discrete spectrum.
\newblock {\em Comm. Math. Phys.}, 94(4):523--535, 1984.

\bibitem{ID-EL-1983}
I.~Daubechies and E.~H. Lieb.
\newblock One-electron relativistic molecules with {C}oulomb interaction.
\newblock {\em Comm. Math. Phys.}, 90(4):497--510, 1983.

\bibitem{Demuth_Casteren}
M.~Demuth and J.~A. van Casteren.
\newblock {\em Stochastic spectral theory for selfadjoint {F}eller operators}.
\newblock Probability and its Applications. Birkh\"{a}user Verlag, Basel, 2000.
\newblock A functional integration approach.

\bibitem{NIST}
{\it NIST Digital Library of Mathematical Functions}.
\newblock http://dlmf.nist.gov/, Release 1.1.2 of 2021-06-15.
\newblock F.~W.~J. Olver, A.~B. {Olde Daalhuis}, D.~W. Lozier, B.~I. Schneider,
  R.~F. Boisvert, C.~W. Clark, B.~R. Miller, B.~V. Saunders, H.~S. Cohl, and
  M.~A. McClain, eds.

\bibitem{CF-RL-1986}
C.~Fefferman and R.~de~la Llave.
\newblock Relativistic stability of matter. {I}.
\newblock {\em Rev. Mat. Iberoamericana}, 2(1-2):119--213, 1986.

\bibitem{RF-EL-RS-2007}
R.~L. Frank, E.~H. Lieb, and R.~Seiringer.
\newblock Stability of relativistic matter with magnetic fields for nuclear
  charges up to the critical value.
\newblock {\em Comm. Math. Phys.}, 275(2):479--489, 2007.

\bibitem{RF-RS-2008}
R.~L. Frank and R.~Seiringer.
\newblock Non-linear ground state representations and sharp {H}ardy
  inequalities.
\newblock {\em J. Funct. Anal.}, 255(12):3407--3430, 2008.

\bibitem{Fuk}
M.~Fukushima, Y.~Oshima, and M.~Takeda.
\newblock {\em Dirichlet forms and symmetric {M}arkov processes}, volume~19 of
  {\em De Gruyter Studies in Mathematics}.
\newblock Walter de Gruyter \& Co., Berlin, extended edition, 2011.

\bibitem{TG-KK-PS-2022}
T.~Grzywny, K.~Kaleta, and P.~Sztonyk.
\newblock Heat kernels of non-local {S}chr{\"o}dinger operators with {K}ato
  potentials.
\newblock preprint 2022, arXiv:2204.04239.

\bibitem{TG-KS-1}
T.~Grzywny and K.~Szczypkowski.
\newblock L\'{e}vy processes: concentration function and heat kernel bounds.
\newblock {\em Bernoulli}, 26(4):3191--3223, 2020.

\bibitem{TG-KS-2}
T.~Grzywny and K.~Szczypkowski.
\newblock Estimates of heat kernels of non-symmetric {L}\'{e}vy processes.
\newblock {\em Forum Math.}, 33(5):1207--1236, 2021.

\bibitem{IH-1977}
I.~W. Herbst.
\newblock Spectral theory of the operator {$(p^{2}+m^{2})^{1/2}-Ze^{2}/r$}.
\newblock {\em Comm. Math. Phys.}, 53(3):285--294, 1977.

\bibitem{Jacob}
N.~Jacob.
\newblock {\em Pseudo differential operators and {M}arkov processes. {V}ol.
  {I}, {II}, {III}}.
\newblock Imperial College Press, London, 2001-2005.

\bibitem{TJ-KK-KS-2022}
T.~Jakubowski, K.~Kaleta, and K.~Szczypkowski.
\newblock Relativistic stable operators with critical potential.
\newblock preprint 2022.

\bibitem{KK-JL-2016}
K.~Kaleta and J.~L\H{o}rinczi.
\newblock Fall-off of eigenfunctions for non-local {S}chr\"{o}dinger operators
  with decaying potentials.
\newblock {\em Potential Anal.}, 46(4):647--688, 2017.

\bibitem{KK-JL-2020}
K.~Kaleta and J.~L\H{o}rinczi.
\newblock Zero-energy bound state decay for non-local {S}chr\"{o}dinger
  operators.
\newblock {\em Comm. Math. Phys.}, 374(3):2151--2191, 2020.

\bibitem{KK-RS-2020}
K.~Kaleta and R.~L. Schilling.
\newblock Progressive intrinsic ultracontractivity and heat kernel estimates
  for non-local {S}chr\"{o}dinger operators.
\newblock {\em J. Funct. Anal.}, 279(6):108606, 69, 2020.

\bibitem{KK-PS-2017}
K.~Kaleta and P.~Sztonyk.
\newblock Small-time sharp bounds for kernels of convolution semigroups.
\newblock {\em J. Anal. Math.}, 132:355--394, 2017.

\bibitem{VK-RS-2013}
V.~Knopova and R.~L. Schilling.
\newblock A note on the existence of transition probability densities of
  {L}\'{e}vy processes.
\newblock {\em Forum Math.}, 25(1):125--149, 2013.

\bibitem{MK-2017}
M.~Kwa\'{s}nicki.
\newblock Ten equivalent definitions of the fractional {L}aplace operator.
\newblock {\em Fract. Calc. Appl. Anal.}, 20(1):7--51, 2017.

\bibitem{EL-RS-2010}
E.~H. Lieb and R.~Seiringer.
\newblock {\em The stability of matter in quantum mechanics}.
\newblock Cambridge University Press, Cambridge, 2010.

\bibitem{EL-YH-1988}
E.~H. Lieb and H.-T. Yau.
\newblock The stability and instability of relativistic matter.
\newblock {\em Comm. Math. Phys.}, 118(2):177--213, 1988.

\bibitem{FN-1986}
F.~Nardini.
\newblock Exponential decay for the eigenfunctions of the two-body relativistic
  {H}amiltonian.
\newblock {\em J. Analyse Math.}, 47:87--109, 1986.

\bibitem{FN-1988}
F.~Nardini.
\newblock On the asymptotic behaviour of the eigenfunctions of the relativistic
  {$N$}-body {S}chr\"{o}dinger operator.
\newblock {\em Boll. Un. Mat. Ital. A (7)}, 2(3):365--369, 1988.

\bibitem{MR-BS-1978}
M.~Reed and B.~Simon.
\newblock {\em Methods of modern mathematical physics. {IV}. {A}nalysis of
  operators}.
\newblock Academic Press [Harcourt Brace Jovanovich, Publishers], New
  York-London, 1978.

\bibitem{MR-2002}
M.~Ryznar.
\newblock Estimates of {G}reen function for relativistic {$\alpha$}-stable
  process.
\newblock {\em Potential Anal.}, 17(1):1--23, 2002.

\bibitem{RS-RS-ZV-2012}
R.~L. Schilling, R.~Song, and Z.~Vondra\v{c}ek.
\newblock {\em Bernstein functions}, volume~37 of {\em De Gruyter Studies in
  Mathematics}.
\newblock Walter de Gruyter \& Co., Berlin, second edition, 2012.
\newblock Theory and applications.

\bibitem{Schmudgen}
K.~Schm\"{u}dgen.
\newblock {\em Unbounded self-adjoint operators on {H}ilbert space}, volume 265
  of {\em Graduate Texts in Mathematics}.
\newblock Springer, Dordrecht, 2012.

\bibitem{BS-1982}
B.~Simon.
\newblock Schr\"{o}dinger semigroups.
\newblock {\em Bull. Amer. Math. Soc. (N.S.)}, 7(3):447--526, 1982.

\bibitem{vanCasteren}
J.~A. van Casteren.
\newblock Generators of strongly continuous semigroups.
\newblock Pitman, Boston, 1985.

\bibitem{Weder-1975}
R.~A. Weder.
\newblock Spectral analysis of pseudodifferential operators.
\newblock {\em J. Functional Analysis}, 20(4):319--337, 1975.

\end{thebibliography}
\end{document}